\documentclass[12pt,english]{article}
\usepackage[latin9]{inputenc}
\usepackage{geometry, mathtools, amsmath, amsthm, amssymb, microtype, hyperref, colortbl, multirow, float}
\makeatletter

\newcommand{\leqnomode}{\tagsleft@true}
\newcommand{\reqnomode}{\tagsleft@false}

\theoremstyle{plain}
\newtheorem{thm}{\protect\theoremname}[section]
\theoremstyle{plain}
\newtheorem{lem}[thm]{\protect\lemmaname}
\theoremstyle{plain}
\newtheorem{cor}[thm]{\protect\corollaryname}
\theoremstyle{plain}
\newtheorem{prop}[thm]{\protect\propositionname}
\theoremstyle{definition}
\newtheorem{problem}[thm]{\protect\problemname}
\theoremstyle{plain}
\newtheorem{conjecture}[thm]{\protect\conjecturename}

\renewcommand{\S}{\mathfrak{S}}
\renewcommand{\P}{\mathcal{P}}
\newcommand{\Q}{\mathcal{Q}}
\newcommand{\pl}{\tilde p}
\newcommand{\ql}{\tilde q}
\newcommand{\Cl}{\tilde C}
\newcommand{\Fl}{\tilde F}

\DeclareMathOperator{\des}{des}

\DeclareMathOperator{\pk}{pk}
\DeclareMathOperator{\lpk}{lpk}

\DeclareMathOperator{\plat}{plat}
\DeclareMathOperator{\ascplat}{ascplat}
\DeclareMathOperator{\lascplat}{lascplat}
\DeclareMathOperator{\lplat}{lplat}
\DeclareMathOperator{\ides}{ides}
\DeclareMathOperator{\ipk}{ipk}
\DeclareMathOperator{\ilpk}{ilpk}
\DeclareMathOperator{\occ}{occ}
\DeclareMathOperator{\mk}{mk}
\DeclareMathOperator{\std}{std}
\DeclareMathOperator{\inv}{inv}
\DeclareMathOperator{\Fib}{Fib}

\usepackage{titlesec}
\titleformat{\section}
  {\normalfont\large\bfseries}{\thesection.}{.8ex}{}
\titleformat{\subsection}
  {\normalfont\bfseries}{\thesubsection.}{.8ex}{}

\let\originalleft\left
\let\originalright\right
\renewcommand{\left}{\mathopen{}\mathclose\bgroup\originalleft}
\renewcommand{\right}{\aftergroup\egroup\originalright}

\usepackage{authblk}
\title{Statistics on clusters and $r$-Stirling permutations}
\author[1]{Sergi Elizalde\thanks{\tt{sergi.elizalde@dartmouth.edu}}}
\author[2]{Justin M. Troyka\thanks{\tt{jtroyka@calstatela.edu}}}
\author[3]{Yan Zhuang\thanks{\tt{yazhuang@davidson.edu}}}
\affil[1]{Department of Mathematics, Dartmouth College} 
\affil[2]{Department of Mathematics, California State University, Los Angeles}
\affil[3]{Department of Mathematics and Computer Science, Davidson College}
\date{\vspace{-0.15in}May 17, 2023}

\makeatother

\providecommand{\corollaryname}{Corollary}
\providecommand{\lemmaname}{Lemma}
\providecommand{\theoremname}{Theorem}
\providecommand{\propositionname}{Proposition}
\providecommand{\problemname}{Problem}
\providecommand{\conjecturename}{Conjecture}

\begin{document}

\maketitle

\begin{abstract}
The Goulden--Jackson cluster method, adapted to permutations by Elizalde and Noy, reduces the problem of counting permutations by occurrences of a prescribed consecutive pattern to that of counting clusters, which are special permutations with a lot of structure. Recently, Zhuang found a generalization of the cluster method which specializes to refinements by additional permutation statistics, namely the inverse descent number $\ides$, the inverse peak number $\ipk$, and the inverse left peak number $\ilpk$. Continuing this line of work, we study the enumeration of $2134\cdots m$-clusters by $\ides$, $\ipk$, and $\ilpk$, which allows us to derive formulas for counting permutations by occurrences of the consecutive pattern $2134\cdots m$ jointly with each of these statistics. Analogous results for the pattern $12\cdots (m-2)m(m-1)$ are obtained via symmetry arguments. Along the way, we discover that $2134\cdots (r+1)$-clusters are equinumerous with $r$-Stirling permutations introduced by Gessel and Stanley, and we establish some joint equidistributions between these two families of permutations.
\end{abstract}
\textbf{\small{}Keywords: }{\small{}consecutive pattern, cluster, descent, peak, $r$-Stirling permutation, plateau}{\let\thefootnote\relax}


\section{Introduction}

Let $\mathfrak{S}_{n}$ denote the symmetric group of permutations of the set $[n]\coloneqq\{1,2,\dots,n\}$. We write permutations in one-line notation---that is, $\pi=\pi_{1}\pi_{2}\cdots\pi_{n}$---and call the $\pi_{i}$ \textit{letters} of $\pi$. The \textit{length} of $\pi$ is the number of letters in $\pi$, so that $\pi$ has length $n$ whenever $\pi\in\mathfrak{S}_{n}$.

The present paper is motivated by the study of consecutive patterns in permutations. Given a sequence $w$ of distinct integers, the \textit{standardization} of $w$---denoted $\std(w)$---is defined to be the permutation obtained by replacing the smallest letter of $w$ with 1, the second smallest letter with 2, and so on. For example, we have $\std(51629)=31425$. Given permutations $\pi\in\mathfrak{S}_{n}$ and $\sigma\in\mathfrak{S}_{m}$, we say that $\pi$ \textit{contains} $\sigma$ (as a \textit{consecutive pattern}) if $\std(\pi_{i}\pi_{i+1}\cdots\pi_{i+m-1})=\sigma$ for some $i\in[n-m+1]$; when this occurs, we call $\pi_{i}\pi_{i+1}\cdots\pi_{i+m-1}$ an \textit{occurrence} of $\sigma$ (as a consecutive pattern) in $\pi$. For instance, the permutation $162437598$ has two occurrences of the consecutive pattern $213$---namely $437$ and $759$---whereas the permutation $672489153$ has no occurrences of $213$.

Let $\occ_{\sigma}(\pi)$ denote the number of occurrences of $\sigma$ in $\pi$. If $\occ_{\sigma}(\pi)=0$, then we say that $\pi$ \textit{avoids} $\sigma$ (as a consecutive pattern). We let $\mathfrak{S}_{n}(\sigma)$ denote the subset of permutations in $\mathfrak{S}_{n}$ avoiding $\sigma$. As observed above, we have $672489153\in\mathfrak{S}_{9}(213)$. For the remainder of this paper, the notions of occurrence and avoidance of patterns in permutations always refer to consecutive patterns unless otherwise stated.

The study of consecutive patterns in permutations was initiated by Elizalde and Noy \cite{Elizalde2003} in 2003, as a variation of classical patterns in permutations, whose systematic study originated in the seminal paper of Simion and Schmidt \cite{Simion1985}. Among the standard tools used in the study of consecutive patterns is the cluster method of Goulden and Jackson. In its original form, the Goulden\textendash Jackson cluster method \cite{Goulden1979} provides a very general formula expressing the generating function for words by occurrences of prescribed subwords in terms of a ``cluster generating function'', which is easier to compute. In 2012, Elizalde and Noy \cite{Elizalde2012} adapted the Goulden\textendash Jackson cluster method to the setting of permutations, which in a similar vein reduces the problem of counting permutations by occurrences of a prescribed consecutive pattern $\sigma$ to that of counting simpler permutations called $\sigma$-clusters.

Work in recent years have led to refinements and variations of the cluster method which keep track of additional permutation statistics. The first result in this direction is a $q$-analogue which provides a refinement by the inversion number $\inv$. This $q$-cluster method first appeared in a survey on consecutive patterns by Elizalde \cite{Elizalde2016}, and was used by Crane, DeSalvo, and Elizalde \cite{Crane2018} in their study of the Mallows distribution. In short, the $q$-cluster method allows us to count permutations by the number of occurrences of $\sigma$ jointly with the statistic $\inv$ if we can count $\sigma$-clusters by $\inv$. 

Subsequently, Zhuang proved a lifting of the cluster method to the Malvenuto\textendash Reutenauer algebra \cite{Zhuang2021}, which we will sometimes refer to as the {\em generalized cluster method}, since it specializes to both Elizalde and Noy's cluster method for permutations and its $q$-analogue upon applying standard homomorphisms. By constructing and applying other homomorphisms, Zhuang obtained further specializations which keep track of other permutation statistics: the \textit{inverse descent number} $\ides$, the \textit{inverse peak number} $\ipk$, and the \textit{inverse left peak number} $\ilpk$. Similar to the $q$-cluster method, these specializations allow us to count permutations by the number of occurrences of $\sigma$ jointly with $\ides$, $\ipk$, or $\ilpk$ if we can count $\sigma$-clusters by the corresponding statistic. This was done in \cite{Zhuang2021} for the monotone patterns $\sigma=12\cdots m$ and $\sigma=m\cdots21$, as well as the patterns $\sigma=12\cdots(a-1)(a+1)a(a+2)(a+3)\cdots m$ where $m\geq5$ and $2\leq a\leq m-2$. Both of these families of patterns are examples of \textit{chain patterns} (see \cite[Section 2.2]{Elizalde2012} for the definition), which is a restrictive condition that makes their clusters more amenable to study. The latter family in particular consists of elementary transpositions, i.e., transpositions of the form $(a,a+1)$; for this reason, we call such patterns \textit{transpositional patterns}. In fact, the work in \cite{Zhuang2021} covers all transpositional patterns except for $1324$ and those of the form $2134\cdots m$ and $12\cdots (m-2)m(m-1)$ for $m\geq3$.

The primary goal of the present paper is to study the enumeration of $2134\cdots m$-clusters by the statistics $\ides$, $\ipk$, and $\ilpk$; via the cluster method, this will allow us to count permutations by occurrences of $2134\cdots m$ jointly with these statistics. By way of reverse-complementation symmetry, we will be able to derive analogous results for the pattern $12\cdots (m-2)m(m-1)$ as well.

\subsection{Inverse \texorpdfstring{$2134\cdots m$}{2134...m}-clusters, $r$-Stirling permutations,
and statistics}

Before describing our work further, let us give some additional definitions and establish notation for our main objects of study. We begin with several permutation statistics: the \textit{descent number} $\des$, the \textit{peak number} $\pk$, the \textit{left peak number} $\lpk$, and their corresponding inverse statistics $\ides$, $\ipk$, and $\ilpk$.
\begin{itemize}
\item We call $i\in[n-1]$ a \textit{descent} of $\pi\in\mathfrak{S}_{n}$ if $\pi_{i}>\pi_{i+1}$. Let $\des(\pi)$ be the number of descents of $\pi$, and $\ides(\pi)\coloneqq\des(\pi^{-1})$ the number of descents of the inverse of $\pi$.
\item We call $i\in\{2,3,\dots,n-1\}$ a \textit{peak} of $\pi\in\mathfrak{S}_{n}$ if $\pi_{i-1}<\pi_{i}>\pi_{i+1}$. Let $\pk(\pi)$ be the number of peaks of $\pi$, and $\ipk(\pi)\coloneqq\pk(\pi^{-1})$.
\item We call $i\in[n-1]$ a \textit{left peak} of $\pi\in\mathfrak{S}_{n}$ if $i$ is a peak of $\pi$, or if $i=1$ and $i$ is a descent of $\pi$. Let $\lpk(\pi)$ be the number of left peaks of $\pi$ (equivalently, the number of peaks of the sequence $0\pi$ obtained by prepending 0 to $\pi$), and $\ilpk(\pi)\coloneqq\ilpk(\pi^{-1})$.
\end{itemize}
As an example, if we have $\pi=4351627$, then $\des(\pi)=3$, $\pk(\pi)=2$, and $\lpk(\pi)=3$. Furthermore, the inverse of $\pi$ is $\pi^{-1}=4621357$, so $\ides(\pi)=2$ and $\ipk(\pi)=\ilpk(\pi)=1$.

Let $m\geq3$ and $k\geq1$. We say that a permutation $\pi\in\mathfrak{S}_{(m-1)k+1}$ is a \textit{$2134\cdots m$-cluster} if $\std(\pi_{i}\pi_{i+1}\pi_{i+m-1})=2134\cdots m$ for all $i\in\{\,(m-1)j+1:0\leq j\leq k-1\,\}$. (We will define clusters in full generality in Section \ref{sec:clusters-def}.) Informally, we can think of a \textit{$2134\cdots m$-cluster} as a permutation consisting of occurrences of $2134\cdots m$, each overlapping with the next one. Clearly, the distributions of $\ides$, $\ipk$, and $\ilpk$ over $2134\cdots m$-clusters are the same as the distributions of $\des$, $\pk$, and $\lpk$ over inverses of $2134\cdots m$-clusters, and it will be easier for us to work with the latter. So, let us define
\[
\P_{r,k}\coloneqq\{\,\pi\in\mathfrak{S}_{rk+1}:\pi^{-1}\text{ is a }2134\cdots(r+1)\text{-cluster}\,\};
\]
the reason for the change from $m$ to $r+1$ will be apparent soon. Continuing the example from the previous paragraph, we have that $\pi=4351627$ is a $213$-cluster, so $\pi^{-1}=4621357$ is a permutation in ${\cal P}_{2,3}$.

Although the study of $2134\cdots m$-clusters was the original goal of this project, our work led us to discover connections between these objects and $r$-Stirling permutations. An \textit{$r$-Stirling permutation}\footnote{Sometimes $r$-Stirling permutations are called $r$\textit{-multipermutations}, such as by Park \cite{Park1994a, Park1994}.} of \textit{order} $k$ is a permutation $\rho=\rho_{1}\rho_{2}\cdots\rho_{rk}$ of the multiset $\{1^{r},2^{r},\dots,k^{r}\}$ (where the superscript denotes multiplicity) such that if $\rho_{i}=\rho_{j}$ and $i<j$, then $\rho_{k}\geq\rho_{i}$ for every $i<k<j$. For example, $\rho=112223444331$ is a 3-Stirling permutation of order 4, but $123244433211$ is not because there is a 2 between two instances of the larger letter 3. Note that 1-Stirling permutations are ordinary permutations; 2-Stirling permutations are simply called \textit{Stirling permutations} and were first studied by Gessel and Stanley in \cite{Gessel1978}, who also proposed the notion of $r$-Stirling permutations for any $r$. There is now a sizable literature on $r$-Stirling permutations and their statistics; see the works listed in \cite{Gessel2020a}.

Let $\Q_{r,k}$ denote the set of $r$-Stirling permutations of order $k$. We define the following statistics on $\rho\in\Q_{r,k}$.
\begin{itemize}
\item We call $i\in[rk-1]$ a \textit{plateau} of $\rho$ if $\rho_{i}=\rho_{i+1}$. Define $\plat(\rho)$ to be the number of plateaus of $\rho$.
\item We call $i\in[rk-1]$ a \textit{leading plateau}\footnote{Leading plateaus are also called \textit{1-plateaus} by Janson, Kuba, and Panholzer \cite{Janson2011}.}  of $\rho\in\Q_{r,k}$ if $i$ is a plateau of $\rho$ and $\rho_{j}\neq\rho_{i}$ for all $j<i$. Define $\lplat(\rho)$ to be the number of leading plateaus of $\rho$.
\item We call $i\in\{2,3,\dots,rk-1\}$ an \textit{ascent-plateau} of $\rho\in\Q_{r,k}$ if $\rho_{i-1}<\rho_{i}=\rho_{i+1}$. Define $\ascplat(\rho)$ to be the number of ascent-plateaus of $\rho$.
\item We call $i\in[rk-1]$ a \textit{left ascent-plateau} of $\rho\in\Q_{r,k}$ if $i$ is an ascent-plateau of $\pi$, or if $i=1$ and $i$ is a plateau of $\pi$. Define $\lascplat(\rho)$ to be the number of left ascent-plateaus of $\rho$.
\end{itemize}
We note that every ascent-plateau is a leading plateau; indeed, if $i$ is an ascent-plateau of $\rho$ and we had $\rho_j=\rho_i$ for some $j<i-1$, then $\rho_{i-1}$ would be between two instances of the larger letter $\rho_i$ and hence $\rho$ would not be an $r$-Stirling permutation.

It is not hard to show that $\left|\P_{r,k}\right|=\left|\Q_{r,k}\right|=\prod_{j=1}^{k-1}(rj+1)$ for all $r\geq2$ and $k\geq1$: the formula for $\left|\Q_{r,k}\right|$ is shown by counting the places to insert $1^r$, then $2^r$, and so on; and the formula for $\left|\P_{r,k}\right|$ is shown in a similar way using the characterization in Proposition \ref{prop:characterizeclusters} given later. Thus, it is worth asking whether any interesting statistics are {\em equidistributed} over these sets, i.e., one statistic has the same distribution over $\P_{r,k}$ as the other does over $\Q_{r,k}$. To that end, define
\begin{align*}
p_{r,k}(i,j) & \coloneqq \left|\{\,\pi\in\P_{r,k}:\des(\pi)=i,\pk(\pi)=j\,\}\right|,\\
\pl_{r,k}(i,j) & \coloneqq\left|\{\,\pi\in\P_{r,k}:\des(\pi)=i,\lpk(\pi)=j\,\}\right|,\\
q_{r,k}(i,j) & \coloneqq\left|\{\,\pi\in\Q_{r,k}:\lplat(\pi)=i,\ascplat(\pi)=j\,\}\right|,\text{ and}\\
\ql_{r,k}(i,j) & \coloneqq\left|\{\,\pi\in\Q_{r,k}:\lplat(\pi)=i,\lascplat(\pi)=j\,\}\right|.
\end{align*}
The following is our main result connecting inverse $2134\cdots m$-clusters and $r$-Stirling permutations.
\begin{thm}[Equidistributions between $\P_{r,k}$ and $\Q_{r,k}$] \label{thm:equidistribution}
For all $r\geq2$, $k\geq 1$, and $i,j\geq0$, we have 
\[ p_{r,k}(i,j)=q_{r,k}(i,j) \qquad \text{and} \qquad \pl_{r,k}(i,j)=\ql_{r,k}(i,j).\]
In other words, the pair $(\des,\pk)$ over $\P_{r,k}$ is jointly equidistributed
with $(\lplat,\ascplat)$ over $\Q_{r,k}$, and the pair $(\des,\lpk)$
over $\P_{r,k}$ is jointly equidistributed with $(\lplat,\lascplat)$
over $\Q_{r,k}$.
\end{thm}

In light of this theorem, for $r\geq2$ and $k\geq1$ we define the polynomials {\allowdisplaybreaks
\begin{align}\label{eq:C-def}
C_{r,k}(t,y) & \coloneqq\sum_{\pi\in\P_{r,k}}t^{\des(\pi)}y^{\pk(\pi)}=\sum_{\pi\in\Q_{r,k}}t^{\lplat(\pi)}y^{\ascplat(\pi)},\\
\label{eq:Cl-def}
\Cl_{r,k}(t,y) & \coloneqq\sum_{\pi\in\P_{r,k}}t^{\des(\pi)}y^{\lpk(\pi)}=\sum_{\pi\in\Q_{r,k}}t^{\lplat(\pi)}y^{\lascplat(\pi)},
\end{align}
}and their specializations {\allowdisplaybreaks
\begin{align}\label{eq:Cdes-def}
C_{r,k}^{\des}(t) & \coloneqq C_{r,k}(t,1)=\sum_{\pi\in\P_{r,k}}t^{\des(\pi)}=\sum_{\rho\in\Q_{r,k}}t^{\lplat(\rho)},\\
\label{eq:Cpk-def}
C_{r,k}^{\pk}(t) & \coloneqq C_{r,k}(1,t)=\sum_{\pi\in\P_{r,k}}t^{\pk(\pi)}=\sum_{\rho\in\Q_{r,k}}t^{\ascplat(\rho)},\text{ and}\\
\label{eq:Clpk-def}
C_{r,k}^{\lpk}(t) & \coloneqq \Cl_{r,k}(1,t)=\sum_{\pi\in\P_{r,k}}t^{\lpk(\pi)}=\sum_{\rho\in\Q_{r,k}}t^{\lascplat(\rho)}.
\end{align}}

\subsection{Outline}

The organization of this paper is as follows. 
In Section~\ref{sec:exact}, we prove a collection of enumerative results for statistics on $\P_{r,k}$ and $\Q_{r,k}$. 
We begin by giving a simple, alternative characterization of permutations in $\P_{r,k}$. Then we show that the numbers $p_{r,k}(i,j)$ and $q_{r,k}(i,j)$ satisfy the same recurrence relation and initial conditions, and similarly with $\pl_{r,k}(i,j)$ and $\ql_{r,k}(i,j)$, thus establishing the equidistributions in Theorem \ref{thm:equidistribution}. We then prove recurrences for the polynomials $C_{r,k}(t,y)$ and $\Cl_{r,k}(t,y)$, and derive differential equations satisfied by the exponential generating functions for these polynomials. Notably, the proof of one of these differential equations requires a decomposition of $r$-Stirling permutations which does not translate easily to inverse $2134\cdots m$-clusters, so our equidistributions play a crucial role there.

In Section~\ref{sec:roots}, we show that the polynomials $C_{r,k}^{\des}(t)$, $C_{r,k}^{\pk}(t)$ and $C_{r,k}^{\lpk}(t) $ have only real roots, which in turn implies that their coefficients are unimodal and log-concave. We also study some probabilistic aspects of these distributions by giving explicit formulas for the average number of descents, peaks, and left peaks over $\P_{r,k}$, and showing that the corresponding distributions are asymptotically normal---that is, they each converge to a normal distribution as $k\rightarrow\infty$.

In Section~\ref{s-clusters}, we apply the cluster method to obtain generating functions for the polynomials 
\begin{align*}
A_{\sigma,n}^{\ides}(s,t) & \coloneqq\sum_{\pi\in\mathfrak{S}_{n}}s^{\occ_{\sigma}(\pi)}t^{\ides(\pi)+1},\\
P_{\sigma,n}^{\ipk}(s,t) & \coloneqq\sum_{\pi\in\mathfrak{S}_{n}}s^{\occ_{\sigma}(\pi)}t^{\ipk(\pi)+1},\text{ and}\\
P_{\sigma,n}^{\ilpk}(s,t) & \coloneqq\sum_{\pi\in\mathfrak{S}_{n}}s^{\occ_{\sigma}(\pi)}t^{\ilpk(\pi)}
\end{align*}
where $n\ge1$ (these are defined to equal 1 when $n=0$), in the special case that $\sigma=2134\cdots m$. We also use symmetry arguments to obtain analogous results for $\sigma=12\cdots (m-2)m(m-1)$. We conclude in Section~\ref{sec:future} with a brief discussion of future directions of research, including a list of conjectures.

\section{Refined enumeration of inverse \texorpdfstring{$2134\cdots m$}{2134...m}-clusters and \texorpdfstring{$r$}{r}-Stirling permutations} \label{sec:exact}

This section focuses on counting permutations in $\P_{r,k}$ and $\Q_{r,k}$ by the statistics defined in the introduction. After describing a convenient characterization of the elements of $\P_{r,k}$, our first goal is to prove the equidistributions in Theorem \ref{thm:equidistribution}. Then we will give recurrence formulas for the polynomials $C_{r,k}(t,y)$ and $\Cl_{r,k}(t,y)$, along with differential equations for their exponential generating functions.

\subsection{A characterization of inverse \texorpdfstring{$2134\cdots m$}{2134...m}-clusters}

In the following proposition, we will see that an inverse $2134\cdots (r+1)$-cluster is characterized by having an increasing sequence of the values that are not congruent to $2$ modulo $r$, along with a relatively free placement of the values that are congruent to $2$ modulo $r$. We will later use this characterization to derive our recursive formulas for the numbers $p_{r,k}(i,j)$ and $\pl_{r,k}(i,j)$. For this proposition only, we will write $\pi(j)$ in place of $\pi_j$ to denote the $j$th entry of $\pi$, so as to streamline the notation when we also use $\pi^{-1}(a)$ to denote the $a$th entry of $\pi^{-1}$.

\begin{prop} \label{prop:characterizeclusters} Let $\pi \in \S_{rk+1}$ for some $r\ge2$ and $k\geq1$. Then $\pi \in \P_{r,k}$ if and only if both of the following two conditions are met:
\begin{itemize}
\item[\normalfont{(i)}] The letters that are not congruent to $2$ modulo $r$ form an increasing subsequence in $\pi$; that is, if $j < j'$, $\pi(j) \not\equiv 2 \pmod{r}$, and $\pi(j') \not\equiv 2 \pmod{r}$, then $\pi(j) < \pi(j')$.
\item[\normalfont{(ii)}] Each letter $ri+2$ that is congruent to $2$ modulo $r$ comes before the letter $ri+1$ in $\pi$; that is, $\pi^{-1}(ri+2) < \pi^{-1}(ri+1)$.
\end{itemize}
\end{prop}

\begin{proof}
First, observe that conditions (i) and (ii) are equivalent to the following conditions on $\pi^{-1}$:
\begin{itemize}
\item[(A)] If $a \in [rk]$ is congruent modulo $r$ to one of $3,4, \ldots, r$, then $\pi^{-1}(a) < \pi^{-1}(a+1)$.
\item[(B)] If $a \in [rk]$ is congruent to $1$ modulo $r$, then $\pi^{-1}(a) < \pi^{-1}(a+2)$ and $\pi^{-1}(a) > \pi^{-1}(a+1)$.
\end{itemize}
So, it suffices to show that $\pi \in \P_{r,k}$ if and only if conditions (A) and (B) are satisfied.

Suppose that $\pi \in \P_{r,k}$, so that $\pi^{-1}$ is a $2134\cdots (r+1)$-cluster. Assume that $a \in [n-1]$ is congruent modulo $r$ to one of $3,4, \ldots, r$. Then the letters in positions $a$ and $a+1$ of $\pi^{-1}$ are both in the $34\cdots (r+1)$ segment of the same occurrence of $2134\cdots(r+1)$, which implies $\pi^{-1}(a) < \pi^{-1}(a+1)$. This proves condition (A).

Now, assume $a \in [n-1]$ is congruent to $1$ modulo $r$. Then the letters in positions $a$, $a+1$, and $a+2$ in $\pi^{-1}$ correspond precisely to the $213$ segment of the same occurrence of $2134\cdots(r+1)$, so $\pi^{-1}(a) < \pi^{-1}(a+2)$ and $\pi^{-1}(a) > \pi^{-1}(a+1)$. This proves condition (B), and completes the proof of one direction of this proposition.

Conversely, let $\pi \in \S_{rk+1}$ be a permutation for which conditions (A) and (B) are satisfied. The order relations in (A) and (B) imply that, for each $i \in [k]$, $\pi^{-1}$ has an occurrence of $2134\cdots(r+1)$ in positions $r(i-1)+1, r(i-1)+2, \ldots, ri+1$. It follows that $\pi^{-1}$ is a $2134\cdots(r+1)$-cluster, so we conclude that $\pi \in \P_{r,k}$.
\end{proof}

\subsection{Equidistributions}

Recall the classical recurrence relation on the \textit{Eulerian numbers} $A(n,i)$:
\[ A(n+1,i) = i\,A(n,i) + (n-i+2)\,A(n,i-1) \]
(see, e.g., \cite[Section 1.4]{Stanley2011}). Since $A(n,i)$ is the number of permutations in $\S_n$ with $i-1$ descents, this recurrence can be proved combinatorially by examining where the letter $n+1$ can be inserted into a permutation in $\S_n$ and whether this insertion creates a new descent.

We will now establish similar recurrences for the numbers $p_{r,k}(i,j)$ and $q_{r,k}(i,j)$, using the same method as for the Eulerian numbers but keeping track of an additional parameter. In fact, the recurrence for $p_{r,k}(i,j)$ is exactly the same as for $q_{r,k}(i,j)$, from which it will follow by induction that $p_{r,k}(i,j) = q_{r,k}(i,j)$, thus proving that the joint distribution of $(\des,\pk)$ over $\P_{r,k}$ is the same as that of $(\lplat,\ascplat)$ over $\Q_{r,k}$. Intuitively, this means that the two statistics are affected in the same way when new highest values are inserted.

Note first that the result clearly holds for $k=1$, since $$\des(2134\cdots (r+1))=1=\lplat(1^r) \quad\text{and}\quad \pk(2134\cdots (r+1))=0=\ascplat(1^r).$$

\begin{prop}\label{prop:recp}
Let $r\ge2$ and $k\ge1$. Then 
$$
p_{r,k+1}(i,j)=jp_{r,k}(i,j)+(i-j+1)p_{r,k}(i,j-1)+(j+1)p_{r,k}(i-1,j)+(rk-i-j+2)p_{r,k}(i-1,j-1).
$$
\end{prop}

\begin{proof}
Every $\pi\in\P_{r,k+1}$ can be obtained uniquely by taking a permutation $\sigma\in\P_{r,k}$ and inserting the letter $rk+2$ in one of $rk+1$ places (any position except the very end), and appending the letters $rk+3,rk+4,\dots,r(k+1)+1$ in this order. Appending $rk+3,rk+4,\dots,r(k+1)+1$ does not change the number of descents or peaks, so the location where the letter $rk+2$ is inserted determines how the statistic values $\des(\pi)$ and $\pk(\pi)$ compare to $\des(\sigma)$ and $\pk(\sigma)$. 

Suppose that $rk+2$ is inserted between $\sigma_\ell$ and $\sigma_{\ell+1}$ (where we set $\ell=0$ if $rk+2$ is inserted at the beginning of $\sigma$), and consider the following four options:
\begin{itemize}
\item if $\ell$ is a peak (and thus also a descent) of $\sigma$, then  $\des(\pi)=\des(\sigma)$ and  $\pk(\pi)=\pk(\sigma)$;
\item if $\ell$ is a descent but not a peak of $\sigma$, then  $\des(\pi)=\des(\sigma)$ and  $\pk(\pi)=\pk(\sigma)+1$;
\item if $\ell+1$ is a peak of $\sigma$ or $\ell=0$, then  $\des(\pi)=\des(\sigma)+1$ and  $\pk(\pi)=\pk(\sigma)$;
\item in all other cases---i.e., $\ell$ is an ascent of $\sigma$ not followed by a descent---we have $\des(\pi)=\des(\sigma)+1$ and  $\pk(\pi)=\pk(\sigma)+1$.
\end{itemize}
Note that the number of positions in the last bullet is $rk-\des(\sigma)-\pk(\sigma)$. Thus, a permutation $\pi\in\P_{r,k+1}$ with $\des(\pi)=i$ and $\pk(\pi)=j$ can be obtained by inserting $rk+2$ in one of the four following ways, prior to appending $rk+3,rk+4,\dots,r(k+1)+1$:
\begin{itemize}
\item by inserting $rk+2$ at one of the $j$ peaks of a permutation $\sigma\in\P_{r,k}$ with $\des(\sigma)=i$ and $\pk(\sigma)=j$;
\item by inserting $rk+2$ at one of the $i-(j-1)$ descents that are not peaks of a permutation $\sigma\in\P_{r,k}$ with $\des(\sigma)=i$ and $\pk(\sigma)=j-1$;
\item by inserting $rk+2$ at the beginning or immediately to the left of one of the $j$ peaks of a permutation $\sigma\in\P_{r,k}$ with $\des(\sigma)=i-1$ and $\pk(\sigma)=j$;
\item by inserting $rk+2$ in any of the $rk-(i-1)-(j-1)$ ascents not followed by a descent of a permutation $\sigma\in\P_{r,k}$ with $\des(\sigma)=i-1$ and $\pk(\sigma)=j-1$.
\end{itemize}
This gives the stated recurrence.
\end{proof}

\begin{prop}\label{prop:recq}
Let $r\ge2$ and $k\ge1$. Then 
$$
q_{r,k+1}(i,j)=jq_{r,k}(i,j)+(i-j+1)q_{r,k}(i,j-1)+(j+1)q_{r,k}(i-1,j)+(rk-i-j+2)q_{r,k}(i-1,j-1).
$$
\end{prop}

\begin{proof}
Every $r$-Stirling permutation $\rho\in\Q_{r,k+1}$ can be obtained uniquely by taking a permutation $\tau\in\Q_{r,k}$ and inserting $r$ consecutive copies of the letter $k+1$ in one of $rk+1$ places. The location where the $k+1$ block is inserted determines how the statistic values $\lplat(\rho)$ and $\ascplat(\rho)$ compare to $\lplat(\tau)$ and $\ascplat(\tau)$.

Suppose that the $k+1$ block is inserted between $\rho_\ell$ and $\rho_{\ell+1}$ (where we set $\ell=0$ if the block is inserted at the beginning of $\rho$, and $\ell=rk$ if it is inserted at the end), and consider the following four options:

\begin{itemize}
\item if $\ell$ is an ascent-plateau (and thus also a leading plateau) of $\tau$, then $\lplat(\rho)=\lplat(\tau)$ and  $\ascplat(\rho)=\ascplat(\tau)$;
\item if $\ell$ is a leading plateau but not an ascent-plateau of $\tau$, then  $\lplat(\rho)=\lplat(\tau)$ and  $\ascplat(\rho)=\ascplat(\tau)+1$;
\item if $\ell+1$ is an ascent-plateau of $\tau$ or $\ell=0$, then  $\lplat(\rho)=\lplat(\tau)+1$ and  $\ascplat(\rho)=\ascplat(\tau)$;
\item in all other cases---i.e., $\ell$ is a descent, $\ell$ is a plateau that is not a leading plateau, $\ell$ is an ascent not followed by a plateau, or $\ell=rk$---we have $\lplat(\rho)=\lplat(\tau)+1$ and  $\ascplat(\rho)=\ascplat(\tau)+1$.
\end{itemize}
Note that the number of positions in the last bullet is $rk-\lplat(\tau)-\ascplat(\tau)$. Thus, a permutation $\rho\in\Q_{r,k+1}$ with $\lplat(\rho)=i$ and $\ascplat(\rho)=j$ can be obtained in one of the four following ways:
\begin{itemize}
\item by inserting $r$ consecutive copies of $k+1$ at one of the $j$ ascent-plateaus of a permutation $\tau\in\Q_{r,k}$ with $\lplat(\tau)=i$ and $\ascplat(\tau)=j$;
\item by inserting $r$ consecutive copies of $k+1$ at one of the $i-(j-1)$ leading plateaus that are not ascent-plateaus of a permutation $\tau\in\Q_{r,k}$ with $\lplat(\tau)=i$ and $\ascplat(\tau)=j-1$;
\item by inserting $r$ consecutive copies of $k+1$ at the beginning or immediately to the left of one of the $j$ ascent-plateaus of a permutation $\tau\in\Q_{r,k}$ with $\lplat(\tau)=i-1$ and $\ascplat(\tau)=j$;
\item by inserting $r$ consecutive copies of $k+1$ in any of the $rk-(i-1)-(j-1)$ positions that are neither leading plateaus, immediately to the left of an ascent-plateau, or at the very beginning of a permutation $\tau\in\Q_{r,k}$ with $\lplat(\tau)=i-1$ and $\ascplat(\tau)=j-1$.
\end{itemize}
This gives the stated recurrence.
\end{proof}

We now complete the proof of Theorem \ref{thm:equidistribution} by showing that the joint distribution of $(\des,\lpk)$ over $\P_{r,k}$ is the same as that of $(\lplat,\lascplat)$ over $\Q_{r,k}$. For $k=1$, we have $$\des(2134\cdots (r+1))=1=\lplat(1^r) \quad\text{and}\quad \lpk(2134\cdots (r+1))=1=\lascplat(1^r).$$
Furthermore, we have the following recurrences for the numbers $\pl(i,j)$ and $\ql(i,j)$.

\begin{prop}\label{prop:recpl}
Let $r\ge2$ and $k\ge1$. Then 
$$
\pl_{r,k+1}(i,j)=j\pl_{r,k}(i,j)+(i-j+1)\pl_{r,k}(i,j-1)+j\pl_{r,k}(i-1,j)+(rk-i-j+3)\pl_{r,k}(i-1,j-1).
$$
\end{prop}

\begin{prop}\label{prop:recql}
Let $r\ge2$ and $k\ge1$. Then 
$$
\ql_{r,k+1}(i,j)=j\ql_{r,k}(i,j)+(i-j+1)\ql_{r,k}(i,j-1)+j\ql_{r,k}(i-1,j)+(rk-i-j+3)\ql_{r,k}(i-1,j-1).
$$
\end{prop}

We omit the proofs for these recurrences as they are nearly identical to the proofs of Propositions \ref{prop:recp} and \ref{prop:recq}. Since $\pl(i,j)$ and $\ql(i,j)$ satisfy the same recurrence, the proof of Theorem \ref{thm:equidistribution} is complete.

\subsection{Polynomial recurrences} \label{sec:polynomial}

Recall from Equations~\eqref{eq:C-def} and~\eqref{eq:Cl-def} that the polynomial $C_{r,k}(t,y)$ encodes the joint distribution of $\des$ and $\pk$ over $\mathcal{P}_{r,k}$---equivalently, that of $\lplat$ and $\ascplat$ over $\mathcal{Q}_{r,k}$---and $\Cl_{r,k}(t,y)$ encodes the joint distribution of $\des$ and $\lpk$ over $\mathcal{P}_{r,k}$---equivalently, that of $\lplat$ and $\lascplat$ over $\mathcal{Q}_{r,k}$. 
For all $r\geq 2$, we have $C_{r,1}(t,y)=t$ and $\Cl_{r,1}(t,y)=ty$. Since
\[
C_{r,k}(t,y) = \sum_{i,j\geq0}p_{r,k}(i,j)t^{i}y^{j}\qquad\text{and}\qquad \Cl_{r,k}(t,y) = \sum_{i,j\geq0}\pl_{r,k}(i,j)t^{i}y^{j},
\]
we may use our recurrences for the numbers $p_{r,k}(i,j)$ and $\pl_{r,k}(i,j)$ to derive recurrences for the polynomials $C_{r,k}(t,y)$ and $\Cl_{r,k}(t,y)$, which can then be used to compute these polynomials for larger values of $k$.  We begin by proving a recurrence for $C_{r,k}(t,y)$.

\begin{thm} \label{t-Crec}
Let $r\geq2$ and $k\geq1$. We have
\begin{align*}
C_{r,k+1}(t,y) & =(1+rky)tC_{r,k}(t,y)+(1-t)ty\frac{\partial}{\partial t}C_{r,k}(t,y)+(1+t)(1-y)y\frac{\partial}{\partial y}C_{r,k}(t,y).
\end{align*}
\end{thm}

\begin{proof}
Multiplying both sides of the recurrence from Proposition \ref{prop:recp} by $t^{i}y^{j}$ and summing over all $i,j\geq1$ yields {\allowdisplaybreaks
\begin{align}
\sum_{i,j\geq1}p_{r,k+1}(i,j)t^{i}y^{j} & =\sum_{i,j\geq1}jp_{r,k}(i,j)t^{i}y^{j}+\sum_{i,j\geq1}(i-j+1)p_{r,k}(i,j-1)t^{i}y^{j}\nonumber \\
 & \quad+\sum_{i,j\geq1}(j+1)p_{r,k}(i-1,j)t^{i}y^{j}+\sum_{i,j\geq1}(rk-i-j+2)p_{r,k}(i-1,j-1)t^{i}y^{j}\nonumber \\
 & =\sum_{i,j\geq1}jp_{r,k}(i,j)t^{i}y^{j}+\sum_{\substack{i\geq1\\
j\geq0
}
}(i-j)p_{r,k}(i,j)t^{i}y^{j+1}\nonumber \\
 & \quad+\sum_{\substack{i\geq0\\
j\geq1
}
}(j+1)p_{r,k}(i,j)t^{i+1}y^{j}+\sum_{i,j\geq0}(rk-i-j)p_{r,k}(i,j)t^{i+1}y^{j+1}.
\label{e-hugemess}
\end{align}}

Let us take the summands in (\ref{e-hugemess}) and rewrite them in terms of $C_{r,k}(t,y)$ and its derivatives. We begin with the left-hand side of (\ref{e-hugemess}). Observe that 
\[
C_{r,k+1}(t,y)-\sum_{\substack{j\geq0}
}p_{r,k+1}(0,j)y^{j}-\sum_{\substack{i\geq1}
}p_{r,k+1}(i,0)t^{i}=\sum_{i,j\geq1}p_{r,k+1}(i,j)t^{i}y^{j}.
\]
We have $p_{r,k+1}(0,j)=0$ for all $j\geq0$ because every permutation in $\mathcal{P}_{r,k}$ has at least one descent. Furthermore, there is only one permutation in $\mathcal{P}_{r,k}$ with no peaks---namely, the permutation beginning with the $k$ letters congruent to 2 modulo $r$ in a decreasing run followed by all remaining letters in an increasing run---which has exactly $k$ descents, so $p_{r,k}(i,0)=\delta_{i,k}$ (the Kronecker delta). Thus, we find
\[
C_{r,k+1}(t,y)-t^{k+1}=\sum_{i,j\geq1}p_{r,k+1}(i,j)t^{i}y^{j}
\]
for the left-hand side of (\ref{e-hugemess}).

We proceed to the right-hand side. We have
\[
y\frac{\partial}{\partial y}C_{r,k}(t,y)=\sum_{i,j\geq0}jp_{r,k}(i,j)t^{i}y^{j},
\]
which implies 
\[
y\frac{\partial}{\partial y}C_{r,k}(t,y)-\sum_{j\geq0}jp_{r,k}(0,j)y^{j}=\sum_{i,j\geq1}jp_{r,k}(i,j)t^{i}y^{j}.
\]
Because $p_{r,k}(0,j)=0$, we obtain
\[
\sum_{i,j\geq1}jp_{r,k}(i,j)t^{i}y^{j}=y\frac{\partial}{\partial y}C_{r,k}(t,y)
\]
for the first summand on the right-hand side of (\ref{e-hugemess}).

Next, we have 
\[
t\frac{\partial}{\partial t}C_{r,k}(t,y)=\sum_{i,j\geq0}ip_{r,k}(i,j)t^{i}y^{j},
\]
so
\[
ty\frac{\partial}{\partial t}C_{r,k}(t,y)-y^{2}\frac{\partial}{\partial y}C_{r,k}(t,y)=\sum_{i,j\geq0}(i-j)p_{r,k}(i,j)t^{i}y^{j+1},
\]
and thus
\[
ty\frac{\partial}{\partial t}C_{r,k}(t,y)-y^{2}\frac{\partial}{\partial y}C_{r,k}(t,y)+\sum_{j\geq0}jp_{r,k}(0,j)y^{j+1}=\sum_{\substack{i\geq1\\
j\geq0
}
}(i-j)p_{r,k}(i,j)t^{i}y^{j+1}.
\]
Because $p_{r,k}(0,j)=0$, we are left with
\[
\sum_{\substack{i\geq1\\
j\geq0
}
}(i-j)p_{r,k}(i,j)t^{i}y^{j+1}=ty\frac{\partial}{\partial t}C_{r,k}(t,y)-y^{2}\frac{\partial}{\partial y}C_{r,k}(t,y)
\]
for the second summand on the right-hand side of (\ref{e-hugemess}).

Now, observe that
\begin{align*}
ty\frac{\partial}{\partial y}C_{r,k}(t,y)+tC_{r,k}(t,y) & =\sum_{i,j\geq0}(j+1)p_{r,k}(i,j)t^{i+1}y^{j},
\end{align*}
which gives us
\[
ty\frac{\partial}{\partial y}C_{r,k}(t,y)+tC_{r,k}(t,y)-\sum_{i\geq0}p_{r,k}(i,0)t^{i+1}=\sum_{\substack{i\geq0\\
j\geq1
}
}(j+1)p_{r,k}(i,j)t^{i+1}y^{j}.
\]
Again using the fact that $p_{r,k}(i,0)=\delta_{i,k}$, we obtain
\[
ty\frac{\partial}{\partial y}C_{r,k}(t,y)+tC_{r,k}(t,y)-t^{k+1}=\sum_{\substack{i\geq0\\
j\geq1
}
}(j+1)p_{r,k}(i,j)t^{i+1}y^{j}
\]
for the third summand on the right-hand side of (\ref{e-hugemess}).

Finally, the last summand on the right-hand side is equal to
\[
rktyC_{r,k}(t,y)-t^{2}y\frac{\partial}{\partial t}C_{r,k}(t,y)-ty^{2}\frac{\partial}{\partial y}C_{r,k}(t,y)
=\sum_{i,j\geq0}(rk-i-j)p_{r,k}(i,j)t^{i+1}y^{j+1}.
\]
Putting everything together, we obtain 
\begin{align*}
C_{r,k+1}(t,y)-t^{k+1} & =y\frac{\partial}{\partial y}C_{r,k}(t,y)+ty\frac{\partial}{\partial t}C_{r,k}(t,y)-y^{2}\frac{\partial}{\partial y}C_{r,k}(t,y)\\
 & \qquad+ty\frac{\partial}{\partial y}C_{r,k}(t,y)+tC_{r,k}(t,y)-t^{k+1}\\
 & \qquad+rktyC_{r,k}(t,y)-t^{2}y\frac{\partial}{\partial t}C_{r,k}(t,y)-ty^{2}\frac{\partial}{\partial y}C_{r,k}(t,y)
\end{align*}
which simplifies to the desired formula.
\end{proof}

Below is an analogous recurrence for the polynomials $\Cl_{r,k}(t,y)$. We omit the proof, as it proceeds in the same way as the proof of Theorem \ref{t-Crec} but uses Proposition \ref{prop:recpl} in place of Proposition \ref{prop:recp}.

\begin{thm} \label{t-Clrec}
Let $r\geq2$ and $k\geq1$. We have
\begin{align*}
\Cl_{r,k+1}(t,y) & =(1+rk)ty\Cl_{r,k}(t,y)+(1-t)ty\frac{\partial}{\partial t}\Cl_{r,k}(t,y)+(1+t)(1-y)y\frac{\partial}{\partial y}\Cl_{r,k}(t,y).
\end{align*}
\end{thm}

By specializing appropriately, we obtain from Theorems \ref{t-Crec} and \ref{t-Clrec} recurrences for the univariate polynomials defined in Equations \eqref{eq:Cdes-def}--\eqref{eq:Clpk-def}.

\begin{cor} \label{c-uniC}
Let $r\geq2$ and $k\geq1$. We have \leqnomode 
\begin{align*}
\text{(a)} && C_{r,k+1}^{\des}(t) & =(1+rk)tC_{r,k}^{\des}(t)+t(1-t)\frac{d}{dt}C_{r,k}^{\des}(t), \\
\text{(b)} && C_{r,k+1}^{\pk}(t) & =(1+rkt)C_{r,k}^{\pk}(t)+2t(1-t)\frac{d}{dt}C_{r,k}^{\pk}(t), \text{ and} \\
\text{(c)} && C_{r,k+1}^{\lpk}(t) & =(1+rk)tC_{r,k}^{\lpk}(t)+2t(1-t)\frac{d}{dt}C_{r,k}^{\lpk}(t).
\end{align*}
\end{cor}

The first few polynomials $C_{r,k}^{\des}(t)$ for $r=2$ and $r=3$ are displayed in Tables~\ref{tab:Cdes2} and~\ref{tab:Cdes3}, respectively. We note that $C_{r,k}^{\des}(t)$ not only encodes the distribution of $\des$ over $\P_{r,k}$ and that of $\lplat$ over $\Q_{r,k}$, but also the distribution over $\Q_{r,k}$ of the statistic $\des$ defined as follows. We say that $i \in [rk]$ is a \textit{descent} of an $r$-Stirling permutation $\rho \in \Q_{r,k}$ if $\rho_i > \rho_{i+1}$ or if $i=rk$ and, as with ordinary permutations, we let $\des(\rho)$ denote the number of descents of $\rho$. It was shown by Janson, Kuba, and Panholzer \cite[Theorems 2 and 8]{Janson2011} that the statistics $\des$ and $\lplat$ are equidistributed over $\Q_{r,k}$, and so
\begin{equation}\label{eq:Cdes-Qdes}
C_{r,k}^{\des}(t)=\sum_{\rho\in\Q_{r,k}}t^{\des(\rho)}.
\end{equation}

\renewcommand{\arraystretch}{1.2}
\begin{table}[H]
\begin{centering}
\begin{tabular}{|c|c|}
\hline 
$k$ & $C_{2,k}^{\des}(t)$\tabularnewline
\hline 
$1$ & $t$\tabularnewline
\hline 
$2$ & $t+2t^{2}$\tabularnewline
\hline 
$3$ & $t+8t^{2}+6t^{3}$\tabularnewline
\hline 
$4$ & $t+22t^{2}+58t^{3}+24t^{4}$\tabularnewline
\hline 
$5$ & $t+52t^{2}+328t^{3}+444t^{4}+120t^{5}$\tabularnewline
\hline 
$6$ & $t+114t^{2}+1452t^{3}+4400t^{4}+3708t^{5}+720t^{6}$\tabularnewline
\hline 
$7$ & $t+240t^{2}+5610t^{3}+32120t^{4}+58140t^{5}+33984t^{6}+5040t^{7}$\tabularnewline
\hline 
\end{tabular}
\par\end{centering}
\caption{Distribution of $\protect\des$ over $\protect\P_{2,k}$ (equivalently,
of $\protect\lplat$ over ${\cal Q}_{2,k}$)}
\label{tab:Cdes2}
\end{table}

\begin{table}[H]
\begin{centering}
\begin{tabular}{|c|c|}
\hline 
$k$ & $C_{3,k}^{\des}(t)$\tabularnewline
\hline 
$1$ & $t$\tabularnewline
\hline 
$2$ & $t+3t^{2}$\tabularnewline
\hline 
$3$ & $t+12t^{2}+15t^{3}$\tabularnewline
\hline 
$4$ & $t+33t^{2}+141t^{3}+105t^{4}$\tabularnewline
\hline 
$5$ & $t+78t^{2}+786t^{3}+1830t^{4}+945t^{5}$\tabularnewline
\hline 
$6$ & $t+171t^{2}+3450t^{3}+17538t^{4}+26685t^{5}+10395t^{6}$\tabularnewline
\hline 
$7$ & $t+360t^{2}+13257t^{3}+125352t^{4}+396495t^{5}+435960t^{6}+135135t^{7}$\tabularnewline
\hline 
\end{tabular}
\par\end{centering}
\caption{Distribution of $\protect\des$ over $\protect\P_{3,k}$ (equivalently,
of $\protect\lplat$ over ${\cal Q}_{3,k}$)}
\label{tab:Cdes3}
\end{table}

In the case of Stirling permutations ($r=2$), the equidistribution result of Janson--Kuba--Panholzer reduces to B\'ona's earlier result \cite[Proposition 1]{Bona2008/09} that $\des$ and $\plat$ are equidistributed over $\Q_{2,k}$, as plateaus and leading plateaus coincide in this setting. Gessel and Stanley studied the distribution of $\des$ over $\Q_{2,k}$ in their seminal paper on Stirling permutations \cite{Gessel1978}; in particular, they proved that 
\[
\frac{C_{2,k}^{\des}(t)}{(1-t)^{2k+1}}=\sum_{m=0}^{\infty}S(k+m,m)t^{m}
\]
where the $S(n,m)$ are Stirling numbers of the second kind, hence the name ``Stirling permutations''.

We display the first few polynomials $C_{r,k}^{\pk}(t)$ and $C_{r,k}^{\lpk}(t)$ for $r=2$ and $r=3$ in Tables~\ref{tab:Cpk2}--\ref{tab:Clpk3}. In terms of $r$-Stirling permutations, these polynomials encode the distributions of $\ascplat$ and $\lascplat$ over $\Q_{r,k}$. 

\begin{table}[H]
\begin{centering}
\begin{tabular}{|c|c|}
\hline 
$k$ & $C_{2,k}^{\pk}(t)$\tabularnewline
\hline 
$1$ & $1$\tabularnewline
\hline 
$2$ & $1+2t$\tabularnewline
\hline 
$3$ & $1+10t+4t^{2}$\tabularnewline
\hline 
$4$ & $1+36t+60t^{2}+8t^{3}$\tabularnewline
\hline 
$5$ & $1+116t+516t^{2}+296t^{3}+16t^{4}$\tabularnewline
\hline 
$6$ & $1+358t+3508t^{2}+5168t^{3}+1328t^{4}+32t^{5}$\tabularnewline
\hline 
$7$ & $1+1086t+21120t^{2}+64240t^{3}+42960t^{4}+5664t^{5}+64t^{6}$\tabularnewline
\hline 
\end{tabular}
\par\end{centering}
\caption{Distribution of $\protect\pk$ over $\protect\P_{2,k}$ (equivalently,
of $\protect\ascplat$ over ${\cal Q}_{2,k}$)}
\label{tab:Cpk2}
\end{table}

\begin{table}[H]
\begin{centering}
\begin{tabular}{|c|c|}
\hline 
$k$ & $C_{3,k}^{\pk}(t)$\tabularnewline
\hline 
$1$ & $1$\tabularnewline
\hline 
$2$ & $1+3t$\tabularnewline
\hline 
$3$ & $1+15t+12t^{2}$\tabularnewline
\hline 
$4$ & $1+54t+165t^{2}+60t^{3}$\tabularnewline
\hline 
$5$ & $1+174t+1365t^{2}+1740t^{3}+360t^{4}$\tabularnewline
\hline 
$6$ & $1+537t+9087t^{2}+27195t^{3}+18900t^{4}+2520t^{5}$\tabularnewline
\hline 
$7$ & $1+1629t+54027t^{2}+317583t^{3}+496440t^{4}+216720t^{5}+20160t^{6}$\tabularnewline
\hline 
\end{tabular}
\par\end{centering}
\caption{Distribution of $\protect\pk$ over $\protect\P_{3,k}$ (equivalently,
of $\protect\ascplat$ over ${\cal Q}_{3,k}$)}
\label{tab:Cpk3}
\end{table}

\begin{table}[H]
\begin{centering}
\begin{tabular}{|c|c|}
\hline 
$k$ & $C_{2,k}^{\lpk}(t)$\tabularnewline
\hline 
$1$ & $t$\tabularnewline
\hline 
$2$ & $2t+t^{2}$\tabularnewline
\hline 
$3$ & $4t+10t^{2}+t^{3}$\tabularnewline
\hline 
$4$ & $8t+60t^{2}+36t^{3}+t^{4}$\tabularnewline
\hline 
$5$ & $16t+296t^{2}+516t^{3}+116t^{4}+t^{5}$\tabularnewline
\hline 
$6$ & $32t+1328t^{2}+5168t^{3}+3508t^{4}+358t^{5}+t^{6}$\tabularnewline
\hline 
$7$ & $64t+5664t^{2}+42960t^{3}+64240t^{4}+21120t^{5}+1086t^{6}+t^{7}$\tabularnewline
\hline 
\end{tabular}
\par\end{centering}
\caption{Distribution of $\protect\lpk$ over $\protect\P_{2,k}$ (equivalently,
of $\protect\lascplat$ over ${\cal Q}_{2,k}$)}
\label{tab:Clpk2}
\end{table}

\begin{table}[H]
\begin{centering}
\begin{tabular}{|c|c|}
\hline 
$k$ & $C_{3,k}^{\lpk}(t)$\tabularnewline
\hline 
$1$ & $t$\tabularnewline
\hline 
$2$ & $2t+2t^{2}$\tabularnewline
\hline 
$3$ & $4t+18t^{2}+6t^{3}$\tabularnewline
\hline 
$4$ & $8t+104t^{2}+144t^{3}+24t^{4}$\tabularnewline
\hline 
$5$ & $16t+504t^{2}+1800t^{3}+1200t^{4}+120t^{5}$\tabularnewline
\hline 
$6$ & $32t+2240t^{2}+16848t^{3}+27600t^{4}+10800t^{5}+720t^{6}$\tabularnewline
\hline 
$7$ & $64t+9504t^{2}+134688t^{3}+439824t^{4}+411600t^{5}+105840t^{6}+5040t^{7}$\tabularnewline
\hline 
\end{tabular}
\par\end{centering}
\caption{Distribution of $\protect\lpk$ over $\protect\P_{3,k}$ (equivalently,
of $\protect\lascplat$ over ${\cal Q}_{3,k}$)}
\label{tab:Clpk3}
\end{table}

Ascent-plateaus and left ascent-plateaus in Stirling permutations (i.e., in the case $r=2$) have been studied previously in \cite{Duh2018, Ma2019, Ma2015, Ma2015a}. Notably, Ma and Mansour \cite{Ma2015} showed that $C_{2,k}^{\pk}(t)$ is precisely the $1/2$\textit{-Eulerian polynomial} $A^{(2)}_k(t)$---a specialization of the $1/r$\textit{-Eulerian polynomial} introduced by Savage and Viswanathan \cite{Savage2012}---and that $C_{2,k}^{\lpk}(t)$ is related to $A^{(2)}_k(t)$ via the equation $C_{2,k}^{\lpk}(t)=t^k A^{(2)}_k(1/t)$. It follows that the coefficients of $C_{2,k}^{\lpk}(t)$ are precisely those of $C_{2,k}^{\pk}(t)$ but in reverse order, as can be observed in Tables~\ref{tab:Cpk2} and~\ref{tab:Clpk2}. These results do not extend to larger values of $r$.\footnote{More precisely, Ma and Mansour \cite{Ma2015} introduced a statistic on $\Q_{r,k}$ called the number of \textit{longest ascent-plateaus}, and they showed that the $1/r$-Eulerian polynomials count permutations in $\Q_{r,k}$ by longest ascent-plateaus. Longest ascent-plateaus are precisely ascent-plateaus when $r=2$, but deviate for larger $r$.}

Below is an explicit formula for the leading coefficients of the polynomials $C_{r,k}^{\pk}(t)$. In Section \ref{sec:conjectures}, we conjecture similar formulas for certain other coefficients of the polynomials $C_{r,k}^{\pk}(t)$ and $C_{r,k}^{\lpk}(t)$.
\begin{prop}
For all $r\geq2$ and $k\geq1$, the number of permutations in $\P_{r,k}$ with exactly $k-1$ peaks---equivalently, the number of permutations in $\Q_{r,k}$ with exactly $k-1$ ascent-plateaus---is equal to $\prod_{j=1}^{k-1}((r-2)j+2)$.
\end{prop}

\begin{proof}
We proceed by induction on $k$. The base case amounts to showing that there is exactly one permutation in $\P_{r,1}$ with 0 peaks, which is easily verified as $2134\cdots(r+1)$ is the only permutation in $\P_{r,1}$. Now, suppose that the result holds for some fixed $k\geq1$. Let $\sigma$ be a permutation in $\mathcal{P}_{r,k}$ with $k-1$ peaks; it suffices to show that there are $(r-2)k+2$ ways to insert the letters $rk+2,rk+3,\dots,r(k+1)+1$ to create a permutation in $\P_{r,k+1}$ with exactly one additional peak. Recall that we must append the letters $rk+3,\dots,r(k+1)+1$ to the end of $\sigma$, and doing so does not change the number of peaks. There are $rk+1$ positions that we may insert $rk+2$ to obtain a permutation in $\P_{r,k+1}$; however, to increase the number of peaks by one, we cannot insert it next to any of the $k-1$ existing peaks, nor can we insert it at the very beginning. This gives us $rk+1-2(k-1)-1=(r-2)k+2$ possibilities, as desired.
\end{proof}

We note that, when $r\geq3$, the quantity $\prod_{j=1}^{k-1}((r-2)j+2)$ also counts metasylvester classes of $(r-2)$-packed words of degree $k$, which are objects defined in \cite{Novelli2020}. 

\subsection{Differential equations}
Next, we derive differential equations satisfied by the exponential generating functions
\[
F_{r}(t,y,x)\coloneqq 1+ \sum_{k=1}^{\infty}C_{r,k}(t,y)\frac{x^{k}}{k!}\qquad\text{and}\qquad \Fl_{r}(t,y,x)\coloneqq 1+ \sum_{k=1}^{\infty}\Cl_{r,k}(t,y)\frac{x^{k}}{k!}
\]
for the polynomials $C_{r,k}(t,y)$ and $\Cl_{r,k}(t,y)$.

It will be convenient to write
\[
F_{r}(t,y,x)=\bar{F}_{r}(t,y,x)+\acute{F}_{r}(t,y,x)\quad\text{and}\quad \Fl_{r}(t,y,x)=y\bar{F}_{r}(t,y,x)+\acute{F}_{r}(t,y,x)
\]
by separating $r$-Stirling permutations that start with a plateau---counted by $\bar{F}_{r}(t,y,x)$---from those that do not---counted by $\acute{F}_{r}(t,y,x)$. We write $F_r\coloneqq F_r(t,y,x)$ and $\Fl_r \coloneqq \Fl_r(t,y,x)$ when it creates no confusion, and similarly with $\bar{F}_{r}$ and $\acute{F}_{r}$. 

All the derivatives in the following theorem and its proof are partial derivatives with respect to $x$, but we will use the prime symbol instead of $\partial/\partial x$ for convenience.

\begin{thm}\label{thm:diffeq}
Let $r\geq2$. The function $\Fl_r$ satisfies the differential equation
\begin{equation}\label{eq:G} \Fl^{\,\prime}_r=\left(t(y-1)+\Fl_r (\Fl_r-1+t)\right)\Fl_r^{\,r-1},\end{equation}
with initial condition $\Fl_r(t,y,0)=1$. Additionally,
\begin{equation}\label{eq:FG}F_r=e^{t(1-y)\int_0^x \Fl_r(t,y,u)^{r-2}\,du}\,\Fl_r.\end{equation}
In particular, for $r=2$, we have $F_2=e^{t(1-y)x} \Fl_2$, which satisfies the differential equation
$$F^{\,\prime}_2=e^{t(y-1)x}\left(e^{t(y-1)x} F_2-1+t\right)F^{\,2}_2,$$
with initial condition $F_2(t,y,0)=1$.
\end{thm}

\begin{proof}
The proof relies on the fact, also used in \cite[Theorem 4.8]{Elizalde2021}, that every nonempty $r$-Stirling permutation $\rho\in\Q_{r,k}$ can be decomposed uniquely as 
\begin{equation}\label{eq:rho_decomposition} \rho=\tau_0 1 \tau_1 1\dots \tau_{r-1} 1 \tau_r,\end{equation}
where the $\tau_i$ are $r$-Stirling permutations on disjoint sets of letters. Indeed, a common letter in $\tau_i$ and $\tau_j$ for $i\neq j$, together with a letter $1$ in between, would mean that $\rho$ is not an $r$-Stirling permutation. Conversely, given a partition of $\{2,3,\dots,k\}$ into $r+1$ disjoint (possibly empty subsets) $S_0,S_1,\dots,S_r$, together with $r$-Stirling permutations $\tau_i$ with their letters relabeled to the elements of $S_i$ for each $i$, the concatenation $\tau_0 1 \tau_1 1\dots \tau_{r-1} 1 \tau_r$ is an $r$-Stirling permutation.

The decomposition~\eqref{eq:rho_decomposition} yields the equation
\begin{equation}\label{eq:F1} \bar{F}_r^{\,\prime}=\left(t+\bar{F}_r(\Fl_r-1+t)\right)\Fl_r^{\,r-1}.\end{equation}
Indeed, since $\rho$ starts with a plateau, either both $\tau_0$ and $\tau_1$ are empty, or $\tau_0$ starts with a plateau. In the first case, the leading plateau $11$ is accounted by the first $t$. In the second case, $\tau_0$ is accounted by $\bar{F}_r$, and $\tau_1$ is accounted by $\Fl_r-1$ if $\tau_1$ is not empty, since a potential initial plateau in $\tau_1$ would create an ascent-plateau in $\rho$, and by $t$ if $\tau_1$ is empty, since it creates a leading plateau $11$ in $\rho$.
Each of the $\tau_i$ for $i\ge2$ is accounted by a $\Fl_r$, since an initial plateau in $\tau_i$ would create an ascent-plateau in $\rho$, but in this case $\tau_i$ being empty does not create a leading plateau in $\rho$.

Similarly, we obtain
\begin{equation}\label{eq:F2} \acute{F}_r^{\,\prime}=(\Fl_r-1+(\acute{F}_r-1)(\Fl_r-1+t))\Fl_r^{\,r-1}.\end{equation}
In this case, since $\rho$ does not start with a plateau, we have two options. If $\tau_0$ is empty, then $\tau_1$ is not, and it is accounted by the first term $\Fl_r-1$. If $\tau_0$ is not empty, then it cannot start with a plateau, so it is accounted by $\acute{F}_r-1$, and in this case $\tau_1$ is accounted by $\Fl_r-1+t$ as before. Again, each of the $\tau_i$ for $i\ge2$ is accounted by a $\Fl_r$.

Multiplying Equation~\eqref{eq:F1} by $y$ and adding it to Equation~\eqref{eq:F2}, we obtain Equation~\eqref{eq:G}. We also obtain $F_r^{\,\prime}=F_r(\Fl_r-1+t)\Fl_r^{\,r-1}$ by adding Equations~\eqref{eq:F1} and~\eqref{eq:F2}. Using Equation~\eqref{eq:G}, this can be rewritten as $F_r^{\,\prime}/F_r=\Fl_r^{\,\prime}/\Fl_r+t(1-y)\Fl_r^{\,r-2}$. Integrating and exponentiating, we obtain Equation~\eqref{eq:FG}.

In the case $r=2$, this expression reduces to $F_2=e^{t(1-y)x}\Fl_2$, and substituting into Equation~\eqref{eq:G} we obtain the differential equation satisfied by $F_2$.
\end{proof}

Note that $F_2(1,y,x)$ is the exponential generating function for the $1/2$-Eulerian polynomials; see the discussion in Section \ref{sec:polynomial}. An expression for the exponential generating function for $1/r$-Eulerian polynomials is given in \cite{Savage2012}.

We note that there does not seem to be a straightforward way to prove Theorem \ref{thm:diffeq} using permutations in $\P_{r,k}$, as our decomposition for $\Q_{r,k}$ does not translate easily. Thus, our equidistribution results in Theorem \ref{thm:equidistribution} are crucial for interpreting Theorem \ref{thm:diffeq} in terms of the joint distributions of $(\des,\pk)$ and $(\des,\lpk)$ over $\P_{r,k}$.

We can also give differential equations for $F_r$ and $\Fl_r$ using our recurrences for the polynomials $C_{r,k}(t,y)$ and $\Cl_{r,k}(t,y)$. Unlike in Theorem \ref{thm:diffeq}, these differential equations involve partial derivatives with respect to all three variables, but one advantage here is that we obtain a differential equation for $F_r$ when $r>2$.

\begin{thm}
Let $r\geq2$. We have \leqnomode 
\begin{alignat*}{1}
\tag{{a}}(1-rtyx)\frac{\partial F_{r}}{\partial x}+(t-1)ty\frac{\partial F_{r}}{\partial t}+(1+t)(y-1)y\frac{\partial F_{r}}{\partial y} & =tF_{r}\quad\text{and}\\
\tag{{b}}(1-rtyx)\frac{\partial \Fl_{r}}{\partial x}+(t-1)ty\frac{\partial \Fl_{r}}{\partial t}+(1+t)(y-1)y\frac{\partial \Fl_{r}}{\partial y} & =ty\Fl_{r}
\end{alignat*}
with initial conditions $F_r(t,y,0)=\Fl_r(t,y,0)=1$.
\end{thm}

\begin{proof}
We only prove part (a); the proof of part (b) is nearly identical.

Taking the recurrence for the polynomials $C_{r,k+1}(t,y)$ in Theorem \ref{t-Crec}, multiplying
both sides by $x^{k}/k!$, and summing over all $k\geq1$ yields {\allowdisplaybreaks
\begin{align*}
\sum_{k=1}^{\infty}C_{r,k+1}(t,y)\frac{x^{k}}{k!} & =\sum_{k=1}^{\infty}(1+rky)tC_{r,k}(t,y)\frac{x^{k}}{k!}+\sum_{k=1}^{\infty}(1-t)ty\frac{\partial}{\partial t}C_{r,k}(t,y)\frac{x^{k}}{k!}\\
 & \qquad+\sum_{k=1}^{\infty}(1+t)(1-y)y\frac{\partial}{\partial y}C_{r,k}(t,y)\frac{x^{k}}{k!}.
\end{align*}}
The left-hand side is given by
\[
\sum_{k=1}^{\infty}C_{r,k+1}(t,y)\frac{x^{k}}{k!}=\sum_{k=0}^{\infty}C_{r,k+1}(t,y)\frac{x^{k}}{k!}-C_{r,1}(t,y)=\frac{\partial F_{r}}{\partial x}-t,
\]
the first summand on the right-hand side is given by {\allowdisplaybreaks 
\begin{align*}
\sum_{k=1}^{\infty}tC_{r,k}(t,y)\frac{x^{k}}{k!}+\sum_{k=1}^{\infty}rktyC_{r,k}(t,y)\frac{x^{k}}{k!} & =t(F_{r}-1)+rytx\sum_{k=1}^{\infty}C_{r,k}(t,y)\frac{x^{k-1}}{(k-1)!}\\
 & =t(F_{r}-1)+rytx\sum_{k=0}^{\infty}C_{r,k+1}(t,y)\frac{x^{k}}{k!}\\
 & =tF_{r}+rytx\frac{\partial F_{r}}{\partial x}-t,
\end{align*}}
the second summand on the right-hand side is given by 
\[
\sum_{k=1}^{\infty}(1-t)ty\frac{\partial}{\partial t}C_{r,k}(t,y)\frac{x^{k}}{k!}=(1-t)ty\frac{\partial F_{r}}{\partial t},
\]
and the final summand on the right-hand side is given by 
\[
\sum_{k=1}^{\infty}(1+t)(1-y)y\frac{\partial}{\partial y}C_{r,k}(t,y)\frac{x^{k}}{k!}=(1+t)(1-y)y\frac{\partial F_{r}}{\partial y}.
\]
Combining all of these terms and rearranging appropriately yields
the desired equation.
\end{proof}

\section{Real-rootedness and asymptotic normality}\label{sec:roots}

In this section, we will work exclusively with the univariate polynomials 
$C_{r,k}^{\des}(t)$, $C_{r,k}^{\pk}(t)$ and  $C_{r,k}^{\lpk}(t)$ defined in Equations~\eqref{eq:Cdes-def}--\eqref{eq:Clpk-def}.

Brenti \cite[Theorem 6.6.3 (i)]{Brenti1989} proved that the polynomials $\sum_{\rho\in\Q_{r,k}}t^{\des(\rho)}$, which coincide with $C_{r,k}^{\des}(t)$ by Equation~\eqref{eq:Cdes-Qdes}, are \textit{real-rooted}, meaning that they have only real roots. This property has combinatorial and probabilistic significance. If a polynomial is real-rooted and all of its coefficients are non-negative, then its sequence of coefficients is log-concave and unimodal. Moreover, if a sequence of polynomials encodes the distribution of a sequence of random variables, then real-rootedness can be used to establish asymptotic normality---that is, the random variables (standardized appropriately) converge in distribution to the standard normal random variable. Indeed, it is known that the distribution of $\des$ (equivalently, $\lplat$) over $\Q_{r,k}$ is asymptotically normal as $k\rightarrow\infty$; this was proved by B\'ona \cite[Theorem 6]{Bona2008/09} in the case $r=2$ and by Janson, Kuba, and Panholzer \cite[Theorem 11]{Janson2011} for arbitrary $r\geq2$.

Here we show that the polynomials $C_{r,k}^{\pk}(t)$ and $C_{r,k}^{\lpk}(t)$ are also real-rooted, and prove that the distributions of $\pk$ and $\lpk$ over $\P_{r,k}$ (and thus those of $\ascplat$ and $\lascplat$ over $\Q_{r,k}$) are asymptotically normal. In fact, for real-rootedness, we will prove the stronger result that the polynomials $\{C_{r,k}^{\pk}(t)\}_{k\geq1}$ and $\{C_{r,k}^{\lpk}(t)\}_{k\geq1}$ are ``generalized Sturm sequences'', which we will define shortly. The approach we use leads to a rederivation of the corresponding results for $\des$ as well.

\subsection{Real-rootedness}\label{sec:realroots}

Let $\mathbb{R}_{\geq0}[t]$ be the subset of $\mathbb{R}[t]$ consisting of polynomials in $t$ with non-negative real coefficients. For two real-rooted polynomials $f,g\in\mathbb{R}[t]$, let $\{r_{i}\}$ be the roots of $f$ and let $\{s_{j}\}$ be the roots of $g$, both in non-increasing order. We say that $g$ \textit{alternates left of} $f$ if $\deg f=\deg g=n$ and 
\[
s_{n}\leq r_{n}\leq\cdots\leq s_{2}\leq r_{2}\leq s_{1}\leq r_{1},
\]
and we say that $g$ \textit{interlaces} $f$ if $\deg f=\deg g+1=n$ and 
\[
r_{n}\leq s_{n-1}\leq\cdots\leq s_{2}\leq r_{2}\leq s_{1}\leq r_{1}.
\]
We write $g\preceq f$ if either $g$ alternates left of $f$ or interlaces $f$. We also adopt the convention that $f\preceq0$ and $0\preceq f$ for any real-rooted $f\in\mathbb{R}[t]$, and that $a\preceq bt+c$ for any $a,b,c\in\mathbb{R}$.

A polynomial in $\mathbb{R}[t]$ is said to be \textit{standard} if it is identically zero or its leading coefficient is positive. (Clearly, all polynomials in $\mathbb{R}_{\geq0}[t]$ are standard.) We call a sequence of standard polynomials $\{f_{n}(t)\}_{n\geq0}$ a \textit{generalized Sturm sequence} if each of the polynomials $f_{n}(t)$ is real-rooted and $f_{0}\preceq f_{1}\preceq f_{2}\preceq\cdots$. The following lemma, which is a special case of a result due to Liu and Wang \cite[Corollary 2.4]{Liu2007}, gives a powerful way of showing that polynomial sequences satisfying a recurrence relation of a certain form are generalized Sturm sequences.
\begin{lem}
\label{l-gss}Let $\{f_{n}(t)\}_{n\geq0}$ be a sequence of polynomials in $\mathbb{R}_{\geq0}[t]$ satisfying the recurrence
\begin{equation}
f_{n}(t)=a_{n}(t)f_{n-1}(t)+b_{n}(t)f_{n-1}^{\prime}(t)\qquad(n\geq1),\label{e-gssrec}
\end{equation}
for some $a_{n}(t),b_{n}(t)\in\mathbb{R}[t]$, and such that $\deg f_{n}=\deg f_{n-1}$ or $\deg f_{n}=\deg f_{n-1}+1$. If $b_{n}(t)\leq0$ whenever $t\leq0$, then $\{f_{n}(t)\}_{n\geq0}$ is a generalized Sturm sequence.
\end{lem}

\begin{thm}
The polynomial sequences $\{C_{r,k}^{\des}(t)\}_{k\geq1}$, $\{C_{r,k}^{\pk}(t)\}_{k\geq1}$, and $\{C_{r,k}^{\lpk}(t)\}_{k\geq1}$ are generalized Sturm sequences \textup{(}and therefore are real-rooted\textup{)} for every $r\geq2$.
\end{thm}

\begin{proof}
We prove the result for $C_{r,k}^{\pk}(t)$; the proofs for $C_{r,k}^{\des}(t)$ and $C_{r,k}^{\lpk}(t)$ are similar.

Fix $r\geq2$. Recall from Corollary \ref{c-uniC} (b) the recurrence 
\[
C_{r,k}^{\pk}(t)=[1+r(k-1)t]C_{r,k-1}^{\pk}(t)+2t(1-t)\frac{d}{dt}C_{r,k-1}^{\pk}(t)\qquad(k\geq2),
\]
which is of the form (\ref{e-gssrec}) up to shifting indices. It is clear from this recurrence (and the initial condition $C_{r,1}^{\pk}(t)=1$) that $\deg C_{r,k}^{\pk}(t)=\deg C_{r,k-1}^{\pk}(t)+1$ for every $k\geq2$. Since $2t(1-t)\leq0$ for all $t\leq0$, we may apply Lemma \ref{l-gss} to conclude that $\{C_{r,k}^{\pk}(t)\}_{k\geq1}$ is a generalized Sturm sequence.
\end{proof}

\subsection{Asymptotic normality}

We now turn our attention to asymptotic normality. Every sequence $\{f_{n}(t)\}_{n\geq0}$ of non-zero polynomials in $\mathbb{R}_{\geq0}[t]$ induces a sequence of random variables $\{X_{n}\}_{n\geq0}$ defined by 
\[
\mathbb{P}(X_{n}=j)=\frac{[t^{j}]f_{n}(t)}{f_{n}(1)}
\]
for all $n,j\geq0$. Note that $\sum_j \mathbb{P}(X_n = j) = 1$. We let $\mu_{n}\coloneqq\mathbb{E}(X_{n})$ denote the mean of $X_{n}$ and $\sigma^{\,2}_{n}\coloneqq\mathbb{V}(X_{n})$ its variance. Let us write $X_{n}\sim{\cal N}(\mu_n,\sigma_n^{2})$ if the standardized random variables $(X_{n}-\mu_n)/\sigma_n$ converge in distribution to the standard normal random variable as $n\rightarrow\infty$; when this occurs, we say that the distribution of $X_{n}$ is \textit{asymptotically normal}. 

The following lemma, due to Bender \cite{Bender1973}, connects real-rootedness to asymptotic normality.
\begin{lem}
\label{l-an1}Let $\{f_{n}(t)\}_{n\geq0}$ be a sequence of non-zero polynomials in $\mathbb{R}_{\geq0}[t]$ and let $\{X_{n}\}_{n\geq0}$ be its corresponding sequence of random variables. If all of the polynomials $f_{n}(t)$ are real-rooted and $\sigma_{n}\rightarrow\infty$ as $n\rightarrow\infty$, then $X_{n}\sim{\cal N}(\mu_{n},\sigma^{\,2}_{n})$.
\end{lem}

In the next lemma we state some results of Hwang, Chern, and Duh \cite[Theorem 1 and Equation (14)]{Hwang2020} concerning sequences $\{f_{n}(t)\}_{n\geq0}$ satisfying a family of recurrences which generalizes the well-known linear recurrence of the Eulerian polynomials. Crucially, not only does this lemma give another way of establishing asymptotic normality, but it also allows us to calculate the mean and give an asymptotic estimate for the variance.
\begin{lem}
\label{l-an2}Let $\{f_{n}(t)\}_{n\geq0}$ be a sequence of non-zero polynomials in $\mathbb{R}_{\geq0}[t]$ and let $\{X_{n}\}_{n\geq0}$ be its corresponding sequence of random variables. Suppose that $f_{n}(t)$ satisfies the recurrence 
\begin{equation}
f_{n}(t)=[\alpha(t)n+\gamma(t)]f_{n-1}(t)+\beta(t)(1-t)f_{n-1}^{\prime}(t)\qquad(n\geq1),\label{e-anrec}
\end{equation}
for some $\alpha(t),\beta(t),\gamma(t)\in\mathbb{R}[t]$. Let us write $\alpha\coloneqq\alpha(1)$, $\beta\coloneqq\beta(1)$, $\gamma\coloneqq\gamma(1)$, and assume $\alpha+2\beta>0$. Furthermore, let
\[
\mu\coloneqq\frac{\alpha^{\prime}(1)}{\alpha+\beta}\quad\text{and}\quad\sigma^{2}\coloneqq\mu+\frac{\alpha^{\prime\prime}(1)-2\mu\beta^{\prime}(1)-\alpha\mu^{2}}{\alpha+2\beta},
\]
and assume $\sigma^{2}>0$.
\begin{enumerate}
\item [\normalfont{(a)}] The mean $\mu_{n}$ of $X_{n}$ satisfies the recurrence 
\[
\mu_{n}=\left(1-\frac{\beta}{\alpha n+\gamma}\right)\mu_{n-1}+\frac{\alpha^{\prime}(1)n+\gamma^{\prime}(1)}{\alpha n+\gamma}\qquad(n\geq1)
\]
with $\mu_{0}=f_{0}^{\prime}(1)/f_{0}(1)$. In addition, we have $\mu_{n}\sim\mu n$.
\item [\normalfont{(b)}] We have $\sigma^{\,2}_{n}\sim\sigma^{2}n$.
\item [\normalfont{(c)}] We have $X_{n}\sim{\cal N}(\mu n,\sigma^{2}n)$.
\end{enumerate}
\end{lem}

To prove the result stated in part (c), the authors use the ``method of moments'', which amounts to showing that the central moments of the random variables $(X_{n}-\mu n)/(\sigma\sqrt{n})$ converge to the central moments of the standard normal random variable; by a classical convergence theorem of Fr\'{e}chet and Shohat \cite{Frechet1931}, this is sufficient to establish asymptotic normality.

Let $\{X_{r,k}^{\des}\}_{k\geq1}$, $\{X_{r,k}^{\pk}\}_{k\geq1}$, and $\{X_{r,k}^{\lpk}\}_{k\geq1}$ be the sequences of random variables corresponding to $\{C_{r,k}^{\des}(t)\}_{k\geq1}$, $\{C_{r,k}^{\pk}(t)\}_{k\geq1}$, and $\{C_{r,k}^{\lpk}(t)\}_{k\geq1}$, respectively.
\begin{prop}
\label{p-meanvar}Let 
\[
c_{r,k}\coloneqq\frac{r\sin(r^{-1}\pi)(\Gamma(1+r^{-1}))^{2}\Gamma(k-r^{-1})}{\pi(r+2)\Gamma(k+r^{-1})}.
\]
Then, for all $r\geq2$ and $k\geq1$, we have {\allowdisplaybreaks 
\begin{align*}
\mathbb{E}(X_{r,k}^{\des}) & =\frac{rk+1}{r+1}\sim\frac{rk}{r+1}, & \mathbb{V}(X_{r,k}^{\des}) & \sim\frac{r^{2}k}{(r+1)^{2}(r+2)},\\
\mathbb{E}(X_{r,k}^{\pk}) & =\frac{(2k-1)r}{2(r+2)}-\frac{r}{2}c_{r,k}\sim\frac{rk}{r+2}, & \mathbb{V}(X_{r,k}^{\pk}) & \sim\frac{2r^{2}k}{(r+2)^{2}(r+4)},\\
\mathbb{E}(X_{r,k}^{\lpk}) & =\frac{rk+1}{r+2}+c_{r,k}\sim\frac{rk}{r+2},\text{ and} & \mathbb{V}(X_{r,k}^{\lpk}) & \sim\frac{2r^{2}k}{(r+2)^{2}(r+4)}.
\end{align*}}
\end{prop}

While the exact formulas for $\mathbb{E}(X_{r,k}^{\pk})$ and $\mathbb{E}(X_{r,k}^{\lpk})$ are fairly complicated expressions involving the gamma function, we point out that in the $r=2$ case they simplify to 
\[
\mathbb{E}(X_{2,k}^{\pk})=\frac{k(k-1)}{2k-1}\quad\text{and}\quad\mathbb{E}(X_{2,k}^{\lpk})=\frac{k^{2}}{2k-1}.
\]
These exact formulas for the means, even if not strictly necessary for proving asymptotic normality (the asymptotic estimates suffice), allow us to locate the mode(s) of the corresponding distributions via Darroch's theorem; see \cite{Darroch1964,Pitman1997} for details.
\begin{proof}[Proof of Proposition \ref{p-meanvar}]
We give the proof for $X_{r,k}^{\pk}$; the results for $X_{r,k}^{\des}$ and $X_{r,k}^{\lpk}$ are obtained in the same way. Observe that the recurrence for the polynomials $C_{r,k}^{\pk}(t)$ is of the form (\ref{e-anrec})---with $\alpha(t)=rt$, $\beta(t)=2t$, and $\gamma(t)=1-rt$---up to shifting indices. Since $\alpha=r$, $\beta=2$, and $\gamma=1-r$, we have $\alpha+2\beta>0$, which gives us $\mu=r/(r+2)$ and 
\[
\sigma^{2}=\frac{r}{r+2}-\frac{4(\frac{r}{r+2})+r(\frac{r}{r+2})^{2}}{r+4}=\frac{2r^{2}}{(r+2)^{2}(r+4)}>0;
\]
hence, the hypotheses of Lemma \ref{l-an2} are satisfied. Applying this lemma, we get our asymptotic estimates for $\mathbb{E}(X_{r,k}^{\pk})$ and $\mathbb{V}(X_{r,k}^{\pk})$. Furthermore, $\mathbb{E}(X_{r,k}^{\pk})$ satisfies the recurrence 
\[
\mathbb{E}(X_{r,k}^{\pk})=\left(1-\frac{2}{r(k-1)+1}\right)\mathbb{E}(X_{r,k-1}^{\pk})+\frac{r(k-1)}{r(k-1)+1}
\]
with initial condition $\mu_{1}=0$; applying the \texttt{rsolve} command in Maple to solve this recurrence yields our exact formula for $\mathbb{E}(X_{r,k}^{\pk})$.
\end{proof}
\begin{thm}
For all $r\geq2$, the distributions of $\des$, $\pk$, and $\lpk$ over $\P_{r,k}$---equivalently, those of $\lplat$, $\ascplat$, and $\lascplat$ over $\Q_{r,k}$---are asymptotically normal.
\end{thm}

\begin{proof}
We have already shown that the polynomials $C_{r,k}^{\des}(t)$, $C_{r,k}^{\pk}(t)$, and $C_{r,k}^{\lpk}(t)$ are real-rooted, and we have 
\begin{align*}
\sqrt{\mathbb{V}(X_{r,k}^{\des})} & \sim\frac{r\sqrt{k}}{(r+1)\sqrt{r+2}}\rightarrow\infty,\\
\sqrt{\mathbb{V}(X_{r,k}^{\pk})} & \sim\frac{r\sqrt{2k}}{(r+2)\sqrt{r+4}}\rightarrow\infty,\text{ and}\\
\sqrt{\mathbb{V}(X_{r,k}^{\lpk})} & \sim\frac{r\sqrt{2k}}{(r+2)\sqrt{r+4}}\rightarrow\infty
\end{align*}
as $k\rightarrow\infty$. Therefore, we may use either Lemma \ref{l-an1} or Lemma \ref{l-an2} (c) to get the desired result.
\end{proof}

\section{Applying the cluster method} \label{s-clusters}

In this section, we apply our results from Section~\ref{sec:roots} along with the generalized cluster method from~\cite{Zhuang2021} to derive formulas for counting permutations in $\S_n$ by occurrences of  $2134\cdots m$ or $12\cdots (m-2)m(m-1)$, jointly with each of the inverse statistics $\ides$, $\ipk$, and $\ilpk$. We begin with an expository overview of clusters, the cluster method, and the reverse-complementation symmetry on permutations.

\subsection{Clusters}\label{sec:clusters-def}

Given $\pi\in\mathfrak{S}_{n}$, $\sigma\in\mathfrak{S}_{m}$, and $S \subseteq [n-m+1]$, we say that $(\pi,S)$ is a \textit{marked permutation} on $\pi$ (with respect to $\sigma$) if $\std(\pi_{i}\pi_{i+2}\cdots\pi_{i+m})=\sigma$ for all $i\in S$. In other words, for each $i\in S$, the permutation $\pi$ has an occurrence of $\sigma$ starting at position $i$. Note that $\pi$ may have other occurrences of $\sigma$ not recorded by $S$. For example, if $\sigma=312$, then $(845132967,\{1,3\})$ is a marked word on $\pi=845132967$ with respect to $\pi$. Let us call the occurrences whose positions are recorded in $S$ \textit{marked occurrences}.

Now, let $\pi\in\mathfrak{S}_{m}$ and $\pi^{\prime}\in\mathfrak{S}_{n}$. We say that $\tau\in\mathfrak{S}_{m+n}$ is a \textit{concatenation} of $\pi$ and $\pi^{\prime}$ if $\std(\tau_{1}\tau_{2}\cdots\tau_{m})=\pi$ and $\std(\tau_{m+1}\tau_{m+2}\cdots\tau_{m+n})=\pi^{\prime}$. Similarly, if we have two marked permutations $(\pi,S)$ and $(\pi^{\prime},S^{\prime})$ with respect to the same pattern $\sigma$, then we say that $(\tau,T)$ is a \textit{concatenation} of $(\pi,S)$ and $(\pi^{\prime},S^{\prime})$ if $\tau$ is a concatenation of $\pi$ and $\pi^{\prime}$, and if $T=S\cup\{\,i+m:i\in S^{\prime}\,\}$. For example, $(618923574,\{2,5,6\})$ is a concatenation of $(2134,\{2\})$ and $(12453,\{1,2\})$, where $\sigma=123$. 

A marked permutation is called a \textit{cluster} if it is not the concatenation of two nonempty marked permutations. A cluster with respect to $\sigma$ is called a \textit{$\sigma$-cluster}. For example, $(618923574,\{2,5,6\})$ is not a $123$-cluster, but $(4351627,\{1,3,5\})$ is a $213$-cluster. 

It is easy to see that, in order for a marked permutation $(\pi,S)$ to be a $\sigma$-cluster, every letter of $\pi$ must belong to at least one marked occurrence in $\pi$, and each marked occurrence (other than the last one) must overlap with the next one. For this reason, it is useful to consider the \textit{overlap set} of a pattern $\sigma\in\mathfrak{S}_{m}$, defined by 
\begin{align*}
O_{\sigma} & \coloneqq \{\,i\in[m-1]:\std(\sigma_{i+1}\sigma_{i+2}\cdots\sigma_{m})=\std(\sigma_{1}\sigma_{2}\cdots\sigma_{m-i})\,\},
\end{align*}
which encodes the positions where two copies of $\sigma$ may overlap. If $O_{\sigma}=\{m-1\}$---that is, if two copies of $\sigma$ cannot overlap in more than one position---then $\sigma$ is called a \textit{non-overlapping pattern}. 

The advantage of working with non-overlapping patterns is that they greatly constrain how clusters can be formed. More precisely, if $\sigma\in\mathfrak{S}_{m}$ is non-overlapping and $(\pi,S)$ is a $\sigma$-cluster, then the length of $\pi$ must be equal to $(m-1)k+1$ where $k$ is the number of marked occurrences of $\sigma$ in $\pi$, and $S=\{\,(m-1)j+1:0\leq j\leq k-1\,\}$. As such, for any non-overlapping pattern $\sigma$, the length of a $\sigma$-cluster determines the positions of its marked occurrences, so we can simply represent the cluster by its underlying permutation. Because the pattern $2134\cdots m$ is non-overlapping, it is now apparent that the definition of $2134\cdots m$-cluster given in the introduction is compatible with the more general definition of cluster given here. The pattern $12\cdots (m-2)m(m-1)$ is non-overlapping as well.

\subsection{The cluster method and its variants}

We are now ready to discuss the Goulden\textendash Jackson cluster method, starting with the adaptation for permutations due to Elizalde and Noy~\cite{Elizalde2012}.

Fix a consecutive pattern $\sigma$. Recall that $\occ_{\sigma}(\pi)$ is the number of occurrences of $\sigma$ in the permutation $\pi$, and let $\mathcal{C}_{\sigma,\pi}$ be the set of all $\sigma$-clusters with underlying permutation $\pi$. If $c$ is a $\sigma$-cluster, let $\mk_{\sigma}(c)$ be the number of marked occurrences in $c$. Define
\begin{align*}
R_{\sigma}(s,x) & \coloneqq\sum_{n=0}^{\infty}\sum_{\pi\in\mathfrak{S}_{n}}\sum_{c\in \mathcal{C}_{\sigma,\pi}}s^{\mk_{\sigma}(c)}\frac{x^{n}}{n!}=\sum_{n=0}^{\infty}\sum_{k=0}^{\infty}r_{\sigma,n,k}s^{k}\frac{x^{n}}{n!}
\end{align*}
where $r_{\sigma,n,k}$ is the number of $\sigma$-clusters of length $n$ with $k$ marked occurrences.
\begin{thm}[Cluster method for permutations]
\label{t-gjcmperm}Let $\sigma$ be a pattern of length at least
2. Then
\[
\sum_{n=0}^{\infty}\sum_{\pi\in\mathfrak{S}_{n}}s^{\occ_{\sigma}(\pi)}\frac{x^{n}}{n!}=(1-x-R_{\sigma}(s-1,x))^{-1}.
\]
\end{thm}

The above formula allows us to count permutations of each length by occurrences of $\sigma$ if we can count $\sigma$-clusters of each length by marked occurrences. In the special case that $\sigma$ is non-overlapping, the number of marked occurrences of $\sigma$ is determined by the length of the underlying permutation, so it suffices to count $\sigma$-clusters by length.

In~\cite{Zhuang2021}, Zhuang obtained a lifting of Theorem \ref{t-gjcmperm} to the Malvenuto\textendash Reutenauer algebra, from which Theorem \ref{t-gjcmperm} and its $q$-analogue (which first appeared in~\cite{Elizalde2016}) can be recovered as special cases, along with specializations that refine by $\ides$, $\ipk$, and $\ilpk$.

To state these specializations, we shall need the \textit{Hadamard product} $*$ on formal power series in $t$, which is defined by
\[
\Big(\sum_{n=0}^{\infty}a_{n}t^{n}\Big)*\Big(\sum_{n=0}^{\infty}b_{n}t^{n}\Big)\coloneqq\sum_{n=0}^{\infty}a_{n}b_{n}t^{n}.
\]
If $f$ is a formal power series in $t$, we denote by 
$$f^{*\left\langle n\right\rangle }\coloneqq\underset{n\text{ times}}{\underbrace{f*\cdots*f}}$$
the $n$-fold Hadamard product of $f$ in the variable $t$. Note that $f$ may have other variables as well, 
but we apply the Hadamard product only for the variable $t$. For example, we have
\[
\left(3xt+\frac{t^{2}}{1-xy}\right)^{*\left\langle n\right\rangle }=(3x)^{n}t+\frac{t^{2}}{(1-xy)^{n}}.
\]

We now state the specialization of the generalized cluster method which refines by the inverse descent number. Recall that 
\begin{align*}
A_{\sigma,n}^{\ides}(s,t) & =\sum_{\pi\in\mathfrak{S}_{n}}s^{\occ_{\sigma}(\pi)}t^{\ides(\pi)+1},
\end{align*}
and let
\[
R_{\sigma,j}^{\ides}(s,t)\coloneqq\sum_{\substack{\pi\in\mathfrak{S}_{j}}
}t^{\ides(\pi)+1}\sum_{c\in \mathcal{C}_{\sigma,\pi}}s^{\mk_{\sigma}(c)}.
\]

\begin{thm}[Cluster method for $\ides$]
\label{t-gjcmides}Let $\sigma$ be a pattern of length at least 2. Then 
\begin{align*}
\sum_{n=0}^{\infty}\frac{A_{\sigma,n}^{\ides}(s,t)}{(1-t)^{n+1}}x^{n} & =\sum_{n=0}^{\infty}\left(\frac{tx}{(1-t)^{2}}+\frac{1}{1-t}\sum_{j=2}^{\infty}R_{\sigma,j}^{\ides}(s-1,t)z^{j}\right)^{*\left\langle n\right\rangle }
\end{align*}
where $z=x/(1-t)$.
\end{thm}

For the inverse peak statistics $\ipk$ and $\ilpk$, recall that
\begin{align*}
P_{\sigma,n}^{\ipk}(s,t)= & \sum_{\pi\in\mathfrak{S}_{n}}s^{\occ_{\sigma}(\pi)}t^{\ipk(\pi)+1}\quad\text{and}\quad P_{\sigma,n}^{\ilpk}(s,t)=\sum_{\pi\in\mathfrak{S}_{n}}s^{\occ_{\sigma}(\pi)}t^{\ilpk(\pi)},
\end{align*}
and define
\[
R_{\sigma,j}^{\ipk}(s,t)\coloneqq\sum_{\substack{\pi\in\mathfrak{S}_{j}}
}t^{\ipk(\pi)+1}\sum_{c\in \mathcal{C}_{\sigma,\pi}}s^{\mk_{\sigma}(c)}\quad\text{and}\quad R_{\sigma,j}^{\ilpk}(s,t)\coloneqq\sum_{\substack{\pi\in\mathfrak{S}_{j}}
}t^{\ilpk(\pi)}\sum_{c\in \mathcal{C}_{\sigma,\pi}}s^{\mk_{\sigma}(c)}.
\]

\begin{thm}[Cluster method for $\ipk$]
\label{t-gjcmipk}Let $\sigma$ be a permutation of length at least 2. Then
\begin{align*}
\frac{1}{1-t}+\frac{1+t}{2(1-t)}\sum_{n=1}^{\infty}P_{\sigma,n}^{\ipk}(s,u)z^{n} & =\sum_{n=0}^{\infty}\left(\frac{2tx}{(1-t)^{2}}+\frac{1+t}{2(1-t)}\sum_{j=2}^{\infty}R_{\sigma,j}^{\ipk}(s-1,u)z^{j}\right)^{*\left\langle n\right\rangle }
\end{align*}
where $u=4t/(1+t)^{2}$ and $z=(1+t)x/(1-t)$.
\end{thm}

\begin{thm}[Cluster method for $\ilpk$]
\label{t-gjcmilpk}Let $\sigma$ be a permutation of length at least 2. Then
\[
\frac{1}{1-t}\sum_{n=0}^{\infty}P_{\sigma,n}^{\ilpk}(s,u)z^{n}=\sum_{n=0}^{\infty}\left(\frac{z}{1-t}+\frac{1}{1-t}\sum_{j=2}^{\infty}R_{\sigma,j}^{\ilpk}(s-1,u)z^{j}\right)^{*\left\langle n\right\rangle }
\]
where $u=4t/(1+t)^{2}$ and $z=(1+t)x/(1-t)$.
\end{thm}

See \cite[Sections 3.5 and 4.3]{Zhuang2021} for discussion of some practical considerations involved in computing the polynomials $A_{\sigma,n}^{\ides}(s,t)$, $P_{\sigma,n}^{\ipk}(s,t)$, and $P_{\sigma,n}^{\ilpk}(s,t)$ from the Hadamard product formulas.

\subsection{Reverse-complementation}

Prior to applying the cluster method, let us briefly discuss how the reverse-complementation symmetry on permutations affects some of the statistics that we consider. Given $\pi\in\mathfrak{S}_{n}$, its \textit{reverse-complement} $\pi^{rc}$ is defined by 
\[
\pi^{rc}=(n+1-\pi_{n})(n+1-\pi_{n-1})\cdots(n+1-\pi_{1}).
\]
We shall use the following proposition to translate our results about the pattern $2134\cdots m$ to results about its reverse-complement $12\cdots (m-2)m(m-1)$. Other symmetries yield some analogous results for the patterns $m\cdots 4312$ and $(m-1)m(m-2)\cdots 21$, but we do not consider them here.

\begin{prop}
\label{p-rc} Let $\pi\in\mathfrak{S}_{n}$ with $n\geq1$. Then
\begin{enumerate}
\item [\normalfont{(a)}] $\ides(\pi^{rc})=\ides(\pi)$,
\item [\normalfont{(b)}] $\lpk(\pi^{rc})=\lpk(\pi)$,
\item [\normalfont{(c)}] $\ilpk(\pi^{rc})=\ilpk(\pi)$, and
\item [\normalfont{(d)}] if $n\geq2$ and $\pi_{n-1}<\pi_{n}$, then $\pk(\pi^{rc})=\lpk(\pi)$,
\end{enumerate}
\end{prop}

\begin{proof}
Recall that the plot of $\pi\in\mathfrak{S}_n$ is the set of points in the plane with coordinates $(i,\pi_i)$ for $1\le i\le n$.
Throughout this proof, we are guided implicitly by the idea that the plot of $\pi^{rc}$ is a half-turn rotation of the plot of $\pi$, and that the plot of $\pi^{-1}$ is a diagonal reflection of the plot of $\pi$. Part (a) is found in \cite[Proposition 2.6 (d)]{Zhuang2021}, so we focus our attention on parts (b)--(d).

Call $i\in \{2,3,\ldots,n\}$ a \textit{right valley} of $\pi\in\mathfrak{S}_{n}$ if $\pi_{i-1}>\pi_{i}<\pi_{i+1}$, or if $i=n$ and $n-1$ is a descent of $\pi$. Let $\operatorname{rval}(\pi)$ be the number of right valleys of $\pi$. We have that $i$ is a left peak of $\pi^{rc}$ if and only if $n+1-i$ is a right valley of $\pi$, so $\lpk(\pi^{rc}) = \operatorname{rval}(\pi)$. We will now show that $\operatorname{rval}(\pi) = \lpk(\pi)$. Every two consecutive left peaks of $\pi$ have exactly one right valley between them, and every two consecutive right valleys of $\pi$ have exactly one left peak between them. Thus, the sequence of positions $\{i: \text{$i$ is a left peak or a right valley of $\pi$}\}$ alternates between left peaks and right valleys, starting with a left peak and ending with a right valley; therefore there is the same number of each. This proves part (b).

Part (c) follows from part (b) using the fact that $(\pi^{rc})^{-1} = (\pi^{-1})^{rc}$. Part (d) follows from part (b) using the fact that $\pi_{n-1} < \pi_n$ implies $\pk(\pi^{rc}) = \lpk(\pi^{rc})$.
\end{proof}

\subsection{Pattern enumeration results}\label{sec:patterns}

We now apply Theorems \ref{t-gjcmides}\textendash \ref{t-gjcmilpk} to $\sigma=2134\cdots m$ to produce the main results of this section.
\begin{thm}
\label{t-231mides} For all $m\geq3$, we have
\begin{align*}
\sum_{n=0}^{\infty}\frac{A_{2134\cdots m,n}^{\ides}(s,t)}{(1-t)^{n+1}}x^{n} & =\sum_{n=0}^{\infty}\left(\frac{tx}{(1-t)^{2}}+\frac{1}{1-t}\sum_{k=1}^{\infty}tC_{m-1,k}^{\des}(t)(s-1)^{k}z^{(m-1)k+1}\right)^{*\left\langle n\right\rangle }
\end{align*}
where $z=x/(1-t)$.
\end{thm}

\begin{proof}
The polynomial $C_{m-1,k}^{\des}(t)$ counts inverse $2134\cdots m$-clusters of length $(m-1)k+1$ by $\des$, or equivalently, $2134\cdots m$-clusters of length $(m-1)k+1$ by $\ides$. Each $2134\cdots m$-cluster must have length $(m-1)k+1$ for some $k\geq1$, and in this case there are exactly $k$ marked occurrences of $2134\cdots m$ in that cluster, so we have
\[
\sum_{j=2}^{\infty}R_{2134\cdots m,j}^{\ides}(s,t)z^{j}=\sum_{k=1}^{\infty}tC_{m-1,k}^{\des}(t)s^{k}z^{(m-1)k+1}.
\]
The result then follows from Theorem \ref{t-gjcmides}.
\end{proof}

\begin{table}[H]
\begin{centering}
\begin{tabular}{|c|c|}
\hline 
$n$ & $A_{213,n}^{\ides}(t)$\tabularnewline
\hline 
$0$ & $1$\tabularnewline
\hline 
$1$ & $t$\tabularnewline
\hline 
$2$ & $t+t^{2}$\tabularnewline
\hline 
$3$ & $t+3t^{2}+t^{3}$\tabularnewline
\hline 
$4$ & $t+7t^{2}+7t^{3}+t^{4}$\tabularnewline
\hline 
$5$ & $t+15t^{2}+32t^{3}+14t^{4}+t^{5}$\tabularnewline
\hline 
$6$ & $t+30t^{2}+123t^{3}+115t^{4}+26t^{5}+t^{6}$\tabularnewline
\hline 
$7$ & $t+57t^{2}+419t^{3}+738t^{4}+361t^{5}+46t^{6}+t^{7}$\tabularnewline
\hline 
$8$ & $t+105t^{2}+1307t^{3}+3983t^{4}+3663t^{5}+1037t^{6}+79t^{7}+t^{8}$\tabularnewline
\hline 
$9$ & $t+190t^{2}+3836t^{3}+18959t^{4}+29824t^{5}+16041t^{6}+2808t^{7}+133t^{8}+t^{9}$\tabularnewline
\hline 
\end{tabular}
\par\end{centering}
\caption{Distribution of $\protect\ides$ over $\mathfrak{S}_{n}(213)$}
\label{tab:ides213}
\end{table}

\begin{table}[H]
\begin{centering}
\begin{tabular}{|c|c|}
\hline 
$n$ & $A_{2134,n}^{\ides}(t)$\tabularnewline
\hline 
$0$ & $1$\tabularnewline
\hline 
$1$ & $t$\tabularnewline
\hline 
$2$ & $t+t^{2}$\tabularnewline
\hline 
$3$ & $t+4t^{2}+t^{3}$\tabularnewline
\hline 
$4$ & $t+10t^{2}+11t^{3}+t^{4}$\tabularnewline
\hline 
$5$ & $t+22t^{2}+60t^{3}+26t^{4}+t^{5}$\tabularnewline
\hline 
$6$ & $t+45t^{2}+251t^{3}+275t^{4}+57t^{5}+t^{6}$\tabularnewline
\hline 
$7$ & $t+89t^{2}+910t^{3}+2000t^{4}+1083t^{5}+120t^{6}+t^{7}$\tabularnewline
\hline 
$8$ & $t+172t^{2}+3034t^{3}+11830t^{4}+12880t^{5}+3889t^{6}+247t^{7}+t^{8}$\tabularnewline
\hline 
$9$ & $t+328t^{2}+9580t^{3}+61504t^{4}+117535t^{5}+72304t^{6}+13159t^{7}+502t^{8}+t^{9}$\tabularnewline
\hline 
\end{tabular}
\par\end{centering}
\caption{Distribution of $\protect\ides$ over $\mathfrak{S}_{n}(2134)$}
\label{tab:ides2134}
\end{table}

We can use Theorem \ref{t-231mides}, together with our recurrence for the polynomials $C_{r,k}^{\des}(t)$ given in Corollary \ref{c-uniC} (a), to compute the polynomials $A_{2134\cdots m,n}^{\ides}(s,t)$. Additionally, using Proposition \ref{p-rc} (a) and the fact that occurrences of a pattern $\sigma$ in $\pi$ directly correspond to occurrences of $\sigma^{rc}$ in $\pi^{rc}$, it follows that the polynomials $A_{2134\cdots m,n}^{\ides}(s,t)$ and $A_{12\cdots (m-2)m(m-1),n}^{\ides}(s,t)$ coincide.

Now, define $A_{\sigma,n}^{\ides}(t)\coloneqq A_{\sigma,n}^{\ides}(0,t)$, and also $P_{\sigma,n}^{\ipk}(t)$ and $P_{\sigma,n}^{\ilpk}(t)$ in the analogous way. These polynomials encode the distributions of $\ides$, $\ipk$, and $\ilpk$ over $\mathfrak{S}_{n}(\sigma)$, the subset of permutations in $\mathfrak{S}_{n}$ avoiding $\sigma$. We display the first few polynomials $A_{213,n}^{\ides}(t)$ and $A_{2134,n}^{\ides}(t)$ in Tables~\ref{tab:ides213} and~\ref{tab:ides2134}.

\begin{thm}
\label{t-231mipk} For all $m\geq3$, we have \leqnomode
\begin{multline*}
\tag{a} \qquad \frac{1}{1-t}+\frac{1+t}{2(1-t)}\sum_{n=1}^{\infty}P_{2134\cdots m,n}^{\ipk}(s,u)z^{n}\\
=\sum_{n=0}^{\infty}\left(\frac{2tx}{(1-t)^{2}}+\frac{1+t}{2(1-t)}\sum_{k=2}^{\infty}uC_{m-1,k}^{\pk}(u)(s-1)^{k}z^{(m-1)k+1}\right)^{*\left\langle n\right\rangle }\qquad
\end{multline*}
and 
\begin{multline*}
\tag{b} \qquad\frac{1}{1-t}+\frac{1+t}{2(1-t)}\sum_{n=1}^{\infty}P_{12\cdots (m-2)m(m-1),n}^{\ipk}(s,u)z^{n}\\
=\sum_{n=0}^{\infty}\left(\frac{2tx}{(1-t)^{2}}+\frac{1+t}{2(1-t)}\sum_{k=2}^{\infty}uC_{m-1,k}^{\lpk}(u)(s-1)^{k}z^{(m-1)k+1}\right)^{*\left\langle n\right\rangle }\qquad
\end{multline*}
where $u=4t/(1+t)^{2}$ and $z=(1+t)x/(1-t)$.
\end{thm}

\begin{proof}
These follow directly from Theorem \ref{t-gjcmipk} and the equations
\[
\sum_{j=2}^{\infty}R_{2134\cdots m,j}^{\ipk}(s,u)z^{j}=\sum_{k=1}^{\infty}uC_{m-1,k}^{\pk}(u)s^{k}z^{(m-1)k+1}
\]
and 
\[
\sum_{j=2}^{\infty}R_{12\cdots (m-2)m(m-1),n}^{\ipk}(s,u)z^{j}=\sum_{k=1}^{\infty}uC_{m-1,k}^{\lpk}(u)s^{k}z^{(m-1)k+1}.
\]
The latter equation is a consequence of Proposition \ref{p-rc} (d) and the observation that every inverse $2134\cdots m$-cluster $\pi$ of length $n$ satisfies $\pi_{n-1} < \pi_n$.
\end{proof}

Setting $s=0$ in the previous theorem yields results for the polynomials $P_{2134\cdots m,n}^{\ipk}(t)$ and $P_{12\cdots (m-2)m(m-1),n}^{\ipk}(t)$ for counting $2134\cdots m$- and $12\cdots (m-2)m(m-1)$-avoiding permutations by $\ipk$. The first few of these polynomials for $m=3$ and $m=4$ are displayed in Tables~\ref{tab:ipk213}--\ref{tab:ipk1243}.

\begin{table}[H]
\begin{centering}
\begin{tabular}{|c|c|c|c|}
\hline 
$n$ & $P_{213,n}^{\ipk}(t)$ & $n$ & $P_{213,n}^{\ipk}(t)$\tabularnewline
\hline 
$0$ & $1$ & $5$ & $5t+48t^{2}+10t^{3}$\tabularnewline
\hline 
$1$ & $t$ & $6$ & $6t+164t^{2}+126t^{3}$\tabularnewline
\hline 
$2$ & $2t$ & $7$ & $7t+522t^{2}+992t^{3}+102t^{4}$\tabularnewline
\hline 
$3$ & $3t+2t^{2}$ & $8$ & $8t+1608t^{2}+6320t^{3}+2240t^{4}$\tabularnewline
\hline 
$4$ & $4t+12t^{2}$ & $9$ & $9t+4880t^{2}+35860t^{3}+29250t^{4}+1794t^{5}$\tabularnewline
\hline 
\end{tabular}
\par\end{centering}
\caption{Distribution of $\protect\ipk$ over $\mathfrak{S}_{n}(213)$}
\label{tab:ipk213}
\end{table}

\begin{table}[H]
\begin{centering}
\begin{tabular}{|c|c|c|c|}
\hline 
$n$ & $P_{2134,n}^{\ipk}(t)$ & $n$ & $P_{2134,n}^{\ipk}(t)$\tabularnewline
\hline 
$0$ & $1$ & $5$ & $12t+82t^{2}+16t^{3}$\tabularnewline
\hline 
$1$ & $t$ & $6$ & $20t+356t^{2}+254t^{3}$\tabularnewline
\hline 
$2$ & $2t$ & $7$ & $33t+1427t^{2}+2496t^{3}+248t^{4}$\tabularnewline
\hline 
$3$ & $4t+2t^{2}$ & $8$ & $54t+5500t^{2}+19756t^{3}+6744t^{4}$\tabularnewline
\hline 
$4$ & $7t+16t^{2}$ & $9$ & $88t+20780t^{2}+138774t^{3}+108752t^{4}+6520t^{5}$\tabularnewline
\hline 
\end{tabular}
\par\end{centering}
\caption{Distribution of $\protect\ipk$ over $\mathfrak{S}_{n}(2134)$}
\label{tab:ipk2134}
\end{table}

\begin{table}[H]
\begin{centering}
\begin{tabular}{|c|c|c|c|}
\hline 
$n$ & $P_{132,n}^{\ipk}(t)$ & $n$ & $P_{132,n}^{\ipk}(t)$\tabularnewline
\hline 
$0$ & $1$ & $5$ & $16t+42t^{2}+5t^{3}$\tabularnewline
\hline 
$1$ & $t$ & $6$ & $32t+184t^{2}+80t^{3}$\tabularnewline
\hline 
$2$ & $2t$ & $7$ & $64t+732t^{2}+770t^{3}+57t^{4}$\tabularnewline
\hline 
$3$ & $4t+t^{2}$ & $8$ & $128t+2752t^{2}+5816t^{3}+1480t^{4}$\tabularnewline
\hline 
$4$ & $8t+8t^{2}$ & $9$ & $256t+9992t^{2}+38212t^{3}+22232t^{4}+1101t^{5}$\tabularnewline
\hline 
\end{tabular}
\par\end{centering}
\caption{Distribution of $\protect\ipk$ over $\mathfrak{S}_{n}(132)$}
\label{tab:ipk132}
\end{table}

\begin{table}[H]
\begin{centering}
\begin{tabular}{|c|c|c|c|}
\hline 
$n$ & $P_{1243,n}^{\ipk}(t)$ & $n$ & $P_{1243,n}^{\ipk}(t)$\tabularnewline
\hline 
$0$ & $1$ & $5$ & $16t+80t^{2}+14t^{3}$\tabularnewline
\hline 
$1$ & $t$ & $6$ & $32t+368t^{2}+230t^{3}$\tabularnewline
\hline 
$2$ & $2t$ & $7$ & $64t+1570t^{2}+2354t^{3}+216t^{4}$\tabularnewline
\hline 
$3$ & $4t+2t^{2}$ & $8$ & $128t+6424t^{2}+19420t^{3}+6082t^{4}$\tabularnewline
\hline 
$4$ & $8t+15t^{2}$ & $9$ & $256t+25664t^{2}+141840t^{3}+101408t^{4}+5746t^{5}$\tabularnewline
\hline 
\end{tabular}
\par\end{centering}
\caption{Distribution of $\protect\ipk$ over $\mathfrak{S}_{n}(1243)$}
\label{tab:ipk1243}
\end{table}

\begin{thm}
\label{t-231milpk} For all $m\geq3$, we have
\begin{align*}
\frac{1}{1-t}\sum_{n=1}^{\infty}P_{2134\cdots m,n}^{\ilpk}(s,u)z^{n} &= \sum_{n=0}^{\infty}\left(\frac{z}{1-t}+\frac{1}{1-t}\sum_{k=1}^{\infty}C_{m-1,k}^{\lpk}(u)(s-1)^{k}z^{(m-1)k+1}\right)^{*\left\langle n\right\rangle }
\end{align*}
where $u=4t/(1+t)^{2}$ and $z=(1+t)x/(1-t)$.
\end{thm}

\begin{proof}
The result follows directly from Theorem \ref{t-gjcmipk} and the equation
\[
\sum_{j=2}^{\infty}R_{2134\cdots m,j}^{\ilpk}(s,u)z^{j}=\sum_{k=1}^{\infty}C_{m-1,k}^{\ilpk}(u)s^{k}z^{(m-1)k+1}.\tag*{\qedhere}
\]
\end{proof}
In light of Proposition \ref{p-rc} (c), we have $P_{12\cdots (m-2)m(m-1),n}^{\ilpk}(s,t)=P_{2134\cdots m,n}^{\ilpk}(s,t)$. The first few polynomials $P_{213,n}^{\ilpk}(t)$ and $P_{2134,n}^{\ilpk}(t)$ are displayed in Tables~\ref{tab:ilpk213}\textendash\ref{tab:ilpk2134}.

\begin{table}[H]
\begin{centering}
\begin{tabular}{|c|c|c|c|}
\hline 
$n$ & $P_{213,n}^{\ilpk}(t)$ & $n$ & $P_{213,n}^{\ilpk}(t)$\tabularnewline
\hline 
$0$ & $1$ & $5$ & $1+33t+29t^{2}$\tabularnewline
\hline 
$1$ & $1$ & $6$ & $1+87t+187t^{2}+21t^{3}$\tabularnewline
\hline 
$2$ & $1+t$ & $7$ & $1+224t+1006t^{2}+392t^{3}$\tabularnewline
\hline 
$3$ & $1+4t$ & $8$ & $1+570t+4880t^{2}+4430t^{3}+295t^{4}$\tabularnewline
\hline 
$4$ & $1+12t+3t^{2}$ & $9$ & $1+1443t+22214t^{2}+39318t^{3}+8817t^{4}$\tabularnewline
\hline 
\end{tabular}
\par\end{centering}
\caption{Distribution of $\protect\ilpk$ over $\mathfrak{S}_{n}(213)$}
\label{tab:ilpk213}
\end{table}

\begin{table}[H]
\begin{centering}
\begin{tabular}{|c|c|c|c|}
\hline 
$n$ & $P_{2134,n}^{\ilpk}(t)$ & $n$ & $P_{2134,n}^{\ilpk}(t)$\tabularnewline
\hline 
$0$ & $1$ & $5$ & $1+52t+57t^{2}$\tabularnewline
\hline 
$1$ & $1$ & $6$ & $1+152t+422t^{2}+55t^{3}$\tabularnewline
\hline 
$2$ & $1+t$ & $7$ & $1+437t+2589t^{2}+1177t^{3}$\tabularnewline
\hline 
$3$ & $1+5t$ & $8$ & $1+1247t+14394t^{2}+15285t^{3}+1127t^{4}$\tabularnewline
\hline 
$4$ & $1+17t+5t^{2}$ & $9$ & $1+3548t+75540t^{2}+156926t^{3}+38899t^{4}$\tabularnewline
\hline 
\end{tabular}
\par\end{centering}
\caption{Distribution of $\protect\ilpk$ over $\mathfrak{S}_{n}(2134)$}
\label{tab:ilpk2134}
\end{table}

\section{Future directions of research}\label{sec:future}

\subsection{The consecutive pattern 1324} \label{sec:1324}

In \cite{Zhuang2021}, the enumeration of permutations by occurrences of a consecutive pattern $\sigma$ jointly with each of the statistics $\ides$, $\ipk$, and $\ilpk$ was done 
for the monotone patterns $\sigma=12\cdots m$ and $\sigma=m\cdots21$, as well as transpositional patterns of the form $12\cdots(a-1)(a+1)a(a+2)(a+3)\cdots m$ where $m\geq5$ and $2\leq a\leq m-2$. In this paper, we have continued this work by considering the transpositional patterns $\sigma=2134\cdots m$ and $\sigma=12\cdots (m-2)m(m-1)$. The only transpositional pattern that has yet to be studied within this line of research is $1324$, so we state the following as an open problem.
\begin{problem}
Count $1324$-clusters by $\ides$, $\ipk$, and $\ilpk$.
\end{problem}

As explained in Section~\ref{sec:clusters-def}, the patterns $2134\cdots m$ and $12\cdots (m-2)m(m-1)$ have the special property of being non-overlapping, which constrains how clusters can be formed. The patterns $12\cdots(a-1)(a+1)a(a+2)(a+3)\cdots m$ (for $m\geq5$ and $2\leq a\leq m-2$) do not have this property, but as mentioned in the introduction, they are examples of chain patterns. Because $1324$ is neither a non-overlapping pattern nor a chain pattern, we suspect that counting $1324$-clusters by $\ides$, $\ipk$, and $\ilpk$ will be more difficult than for the other transpositional patterns.

\subsection{Conjectures} \label{sec:conjectures}

We conclude by posing several conjectures concerning the polynomials that we have studied in this paper.

Our first three conjectures give formulas for the linear coefficients of the polynomials $C_{r,k}^{\pk}(t)$, the leading coefficients of $C_{r,k}^{\lpk}(t)$, and the coefficients of $t^{k-1}$ in $C_{3,k}^{\lpk}(t)$. We have empirically verified these three conjectures for all $r,k \leq 100$.

\begin{conjecture} \label{cj-linCpk}
For all $r\geq2$ and $k\geq1$, the number of permutations in $\P_{r,k}$ with exactly one peak---equivalently, the number of permutations in $\Q_{r,k}$ with exactly one ascent-plateau---is equal to $(3^k-2k-1)r/4$.
\end{conjecture}

The numbers $(3^k-2k-1)r/4$ for $r=3$ and $r=4$ are given by OEIS \cite{oeis} sequences A290764 and A061981, respectively. Notably, the above conjecture implies that $r+1$ times the number of permutations in $\P_{r,k}$ with exactly one peak equals $r$ times the number of permutations in $\P_{r+1,k}$ with exactly one peak, and also that $r$ times the number of permutations in $\P_{2,k}$ with exactly one peak equals twice the number of permutations in $\P_{r,k}$ with exactly one peak. It would be interesting to prove these relations combinatorially.

\begin{conjecture} \label{cj-kClpk}
For all $r\geq2$ and $k\geq1$, the number of permutations in $\P_{r,k}$ with exactly $k$ left peaks---equivalently, the number of permutations in $\Q_{r,k}$ with exactly $k$ left ascent-plateaus---is equal to $\prod_{j=1}^{k-1}((r-2)j+1)$.
\end{conjecture}

Recall that $\left|\P_{r,k}\right|=\left|\Q_{r,k}\right|=\prod_{j=1}^{k-1}(rj+1)$, so an ideal proof of  Conjecture \ref{cj-kClpk} would establish a bijection between $\P_{r-2,k}$ and the set of permutations in $\P_{r,k}$ with exactly $k$ left peaks (or between $\Q_{r-2,k}$ and permutations in $\Q_{r,k}$ with exactly $k$ left ascent-plateaus).

\begin{conjecture} \label{cj-k-1Clpk}
For all $k\geq2$, the number of permutations in $\P_{3,k}$ with exactly $k-1$ left peaks---equivalently, the number of permutations in $\Q_{3,k}$ with exactly $k-1$ left ascent-plateaus---is equal to $k! \binom{k}{2}$.
\end{conjecture}

The numbers $k! \binom{k}{2}$ are given in OEIS sequence A001804 \cite{oeis}. 

Our remaining conjectures are about the polynomials $A_{\sigma,n}^{\ides}(t)$, $P_{\sigma,n}^{\ipk}(t)$, and $P_{\sigma,n}^{\ilpk}(t)$. The first of these conjectures captures the observation that the coefficients of $t^{n-1}$ in $A_{213,n}^{\ides}(t)$ seem to match OEIS sequence A001924 \cite{oeis}, which involves the Fibonacci numbers $\Fib_n$ defined by $\Fib_0=0$, $\Fib_1=1$, and $\Fib_n = \Fib_{n-1} + \Fib_{n-2}$ for $n\geq2$. 

\begin{conjecture}\label{cj-ides}
For all $n\geq 0$, the number of permutations in $\mathfrak{S}_{n}(213)$ whose inverse has exactly $n-2$ descents is equal to $\Fib_{n+3}-n-2$.
\end{conjecture}

Conjecture \ref{cj-ides} has been verified for all $n\leq 100$. The numbers $\Fib_{n+3}-n-2$ are known to count Motzkin paths of length $n$ with exactly one ascent \cite[Corollary 7]{Zhuang2018}, whereas adding one we obtain the number of Fishburn permutations of length $n$ avoiding the classical patterns 321 and 1423 \cite[Proposition 7.12]{Egge2022}. These suggest possibilities for bijective proofs of Conjecture \ref{cj-ides}.

For fixed $m$, the linear coefficients of $P_{2134\cdots m,n}^{\ipk}(t)$ appear as the $m$th column of OEIS triangle A172119 \cite{oeis}. This suggests our next conjecture, which we have verified for all $m\leq 50$ and $n\leq 70$.

\begin{conjecture} \label{cj-ipk}
For all $m\geq 3$ and $n\geq 1$, the number of permutations in $\mathfrak{S}_{n}(2134\cdots m)$ whose inverse has no peaks is equal to 
\[
\sum_{j=0}^{\left\lfloor \frac{n-1}{m-1}\right\rfloor }(-1)^{j}{n-(m-2)j-1 \choose n-(m-1)j-1}2^{n-(m-1)j-1}.
\]
\end{conjecture}

Our final conjecture extends \cite[Conjecture 6.1]{Zhuang2021} and concerns real-rootedness.

\begin{conjecture}
\label{cj-realroots} Let $\sigma$ be any of the following consecutive patterns:
\begin{itemize}
\item $12\cdots m$ or $m\cdots21$ where $m\geq3$;
\item $12\cdots(a-1)(a+1)a(a+2)(a+3)\cdots m$ where $m\geq5$ and $2\leq a\leq m-2$;
\item $2134\cdots m$ or $12\cdots (m-2)m(m-1)$ where $m\geq3$.
\end{itemize}
Then the polynomials $A_{\sigma,n}^{\ides}(t)$, $P_{\sigma,n}^{\ipk}(t)$,
and $P_{\sigma,n}^{\ilpk}(t)$ are real-rooted for all $n\geq2$.
\end{conjecture}

We have verified Conjecture \ref{cj-realroots} for all $m\leq50$ and $n\leq70$. Our generating function formulas in Section~\ref{sec:patterns} (and those in \cite{Zhuang2021}) played a pivotal role in formulating and gathering evidence for this conjecture, as they have allowed us to efficiently compute the polynomials $A_{\sigma,n}^{\ides}(t)$, $P_{\sigma,n}^{\ipk}(t)$, and $P_{\sigma,n}^{\ilpk}(t)$ for the above patterns. On the other hand, these formulas do not easily lead to recurrences for these polynomials, so we are unable to use an interlacing roots argument like in Section~\ref{sec:realroots} to tackle real-rootedness.

\bigskip{}\bigskip{} \noindent \textbf{Acknowledgements.} We thank the anonymous referee for carefully reading our paper and providing thoughtful comments. SE was partially supported by Simons Collaboration Grant \#929653. YZ was partially supported by an AMS--Simons Travel Grant.

\bibliographystyle{plain}
\bibliography{bibliography}

\end{document}